\documentclass[11pt]{article}
\usepackage{hyperref, amsfonts, amsmath, amssymb, amsthm, fullpage, enumitem, tikz, bm, mathtools, microtype, mleftright}
\usepackage[noabbrev, capitalize]{cleveref}

\definecolor{litegray}{RGB}{192,192,192}
\usetikzlibrary{decorations.pathreplacing, calc}
\tikzstyle{vertex}=[circle, draw, fill=white, inner sep=0pt, minimum width=4pt]
\tikzstyle{root-vertex}=[circle, draw, fill=darkgray, inner sep=0pt, minimum width=4pt]
\tikzstyle{circ}=[circle, draw, fill=white, inner sep=1pt]
\tikzstyle{root}=[circle, draw, fill=darkgray, inner sep=1pt, text=white, minimum width=12pt]

\tikzstyle{braket}=[decorate,decoration={brace,amplitude=10pt},xshift=0pt,yshift=-10pt, black]

\title{Beyond the classification theorem of\\Cameron, Goethals, Seidel, and Shult}
\author{
    Hricha Acharya\thanks{School of Mathematical and Statistical Sciences, Arizona State University, Tempe, AZ 85281, USA. Email: {\tt hachary3@asu.edu}.}
    \and Zilin Jiang\thanks{School of Mathematical and Statistical Sciences, and School of Computing and Augmented Intelligence, Arizona State University, Tempe, AZ 85281, USA. Email: {\tt zilinj@asu.edu}. Supported in part by U.S.\ taxpayers through NSF grant DMS-2127650.}
}
\date{}

\newtheorem{theorem}{Theorem}[section]
\newtheorem{corollary}[theorem]{Corollary}
\newtheorem{lemma}[theorem]{Lemma}
\newtheorem{proposition}[theorem]{Proposition}
\newtheorem{problem}[theorem]{Problem}

\theoremstyle{definition}
\newtheorem{definition}[theorem]{Definition}

\theoremstyle{remark}
\newtheorem*{remark}{Remark}
\newtheorem*{example}{Example}
\newtheorem*{claim*}{Claim}

\newenvironment{claimproof}[1][Proof]{\begin{proof}[#1]}{\end{proof}}

\newcommand{\computer}{\tikz[scale=0.28, baseline=0.1ex, transform shape]{\draw[rounded corners=0.02] (0.15, 0.45) rectangle (0.85, 0.85) {};\draw[rounded corners=0.02] (0, 0.3) rectangle (1, 1) {};\draw[rounded corners=0.02] (0, 0) rectangle (1, 0.2) {};\draw (0.2,0.1) -- (0.3,0.1);\draw (0.5,0.1) -- (0.8,0.1);}}
\newenvironment{cproof}[1][Proof]{\begin{proof}[#1]}{\end{proof}}

\DeclarePairedDelimiter\abs{\lvert}{\rvert}

\newcommand{\dset}[2]{\left\{{#1}\colon{#2}\right\}}
\newcommand{\sset}[1]{\left\{{#1}\right\}}

\newcommand{\ga}{\gamma}
\newcommand{\la}{\lambda}
\newcommand{\las}{\lambda^*}

\newcommand{\D}{\mathcal{D}}

\newcommand{\F}{\mathcal{F}}
\newcommand{\G}{\mathcal{G}}
\newcommand{\HH}{\mathcal{H}}
\newcommand{\N}{\mathbb{N}}
\newcommand{\R}{\mathbb{R}}
\newcommand{\X}{\mathcal{X}}

\newcommand{\bx}{\bm{x}}
\newcommand{\tbx}{\tilde{\bm{x}}}
\newcommand{\T}{\intercal}

\newcommand{\etwo}{\tikz[baseline=-0.2ex, scale=0.9]{\draw[darkgray, thick] (0,0)--(0.2,0); \draw[darkgray, fill=darkgray] (0,0) circle (0.05); \draw[darkgray, fill=darkgray] (0.1,0.17320508075) circle (0.05); \draw[darkgray, fill=white] (0.2,0) circle (0.05);}}
\newcommand{\ftwo}{\tikz[baseline=-0.2ex, scale=0.9]{\draw[darkgray, thick] (0,0)--(0.2,0); \draw[darkgray, fill=darkgray] (0,0) circle (0.05); \draw[darkgray, fill=darkgray] (0.1,0.17320508075) circle (0.05); \draw[darkgray, fill=darkgray] (0.2,0) circle (0.05);}}

\newcommand{\stwo}{\tikz[baseline=-0.2ex, scale=0.9]{\draw[darkgray, thick] (0,0)--(0.1,0.17320508075)--(0.2,0); \draw[darkgray, fill=white] (0,0) circle (0.05); \draw[darkgray, fill=darkgray] (0.1,0.17320508075) circle (0.05); \draw[darkgray, fill=white] (0.2,0) circle (0.05);}}

\newcommand{\ktwo}{\tikz[baseline=-0.2ex, scale=0.9]{\draw[darkgray, thick] (0,0)--(0,0.17320508075); \draw[darkgray, fill=darkgray] (0,0) circle (0.05); \draw[darkgray, fill=darkgray] (0,0.17320508075) circle (0.05);}}
\newcommand{\ktwoc}{\tikz[baseline=-0.2ex, scale=0.9]{\draw[darkgray, fill=darkgray] (0,0) circle (0.05); \draw[darkgray, fill=darkgray] (0,0.17320508075) circle (0.05);}}

\newcommand{\ape}[1]{(#1,\etwo)}
\newcommand{\tpe}[1]{(#1,\ftwo)}

\newcommand{\apesub}[1]{(#1,{\tikz[baseline=-0.2ex, scale=0.63]{\draw[darkgray, thick] (0,0)--(0.2,0); \draw[darkgray, fill=darkgray] (0,0) circle (0.05); \draw[darkgray, fill=darkgray] (0.1,0.17320508075) circle (0.05); \draw[darkgray, fill=white] (0.2,0) circle (0.05);}})}
\newcommand{\tpesub}[1]{(#1,{\tikz[baseline=-0.2ex, scale=0.63]{\draw[darkgray, thick] (0,0)--(0.2,0); \draw[darkgray, fill=darkgray] (0,0) circle (0.05); \draw[darkgray, fill=darkgray] (0.1,0.17320508075) circle (0.05); \draw[darkgray, fill=darkgray] (0.2,0) circle (0.05);}})}
\newcommand{\etwographsub}{\tikz[baseline=-0.2ex, scale=0.63]{\draw[darkgray, thick] (0,0)--(0.2,0); \draw[darkgray, fill=white] (0,0) circle (0.05); \draw[darkgray, fill=white] (0.1,0.17320508075) circle (0.05); \draw[darkgray, fill=white] (0.2,0) circle (0.05);}}

\begin{document}

\maketitle

\begin{abstract}
	In 1976, Cameron, Goethals, Seidel, and Shult classified all the graphs whose smallest eigenvalue is at least $-2$ by relating such graphs to root systems that appear in the classification of semisimple Lie algebras. In this paper, extending their beautiful theorem, we give a complete classification of all connected graphs whose smallest eigenvalue lies in $(-\lambda^*, -2)$, where $\lambda^* = \rho^{1/2} + \rho^{-1/2} \approx 2.01980$, and $\rho$ is the unique real root of $x^3 = x + 1$. Our result is the first classification of infinitely many connected graphs with their smallest eigenvalue in $(-\lambda, -2)$ for any constant $\lambda > 2$.
\end{abstract}



\section{Introduction} \label{sec:intro}

A core problem in spectral graph theory is the characterization of graphs with restricted eigenvalues. In this paper, by eigenvalues of a graph $G$, we mean those associated with its adjacency matrix $A_G$. We focus on graphs with eigenvalues bounded from below, and we denote by $\G(\la)$ the family of graphs with smallest eigenvalue at least $-\la$. Because $\G(\la)$ is closed under disjoint union, it is enough to characterize the connected graphs in $\G(\la)$.

The well-known fact that all line graphs have smallest eigenvalue at least $-2$ prompted a great deal of interest in the characterization of graphs in $\G(2)$. After Hoffman~\cite{H69} constructed generalized line graphs, the interest deepened as it became apparent that generalized line graphs are not the only graphs in $\G(2)$. The ubiquitous Petersen graph, the Shrikhande graph, the Clebsch graph, the Schl\"afli graph, and the three Chang graphs were among the first exceptional graphs to be identified. An important result in this direction is the complete enumeration of strongly regular graphs in $\G(2)$ by Seidel~\cite{S73}.

In 1976, the characterization of graphs in $\G(2)$ culminated in a beautiful theorem of Cameron, Goethals, Seidel, and Shult~\cite{CGSS76}. The ingenuity of their approach lies in translating the spectral graph-theoretic problem into one about root systems, which are used in the classification of semisimple Lie algebras. They proved that, apart from the generalized line graphs, there exist only finitely many other connected graphs in $\G(2)$, all having at most $36$ vertices. More precisely, every connected graph in $\G(2)$ can be represented by a root system $D_n$ (for some $n$) or by the root system $E_8$. For a comprehensive account of this theory and related developments, we refer the reader to the monograph \cite{CRS04}.

Can we classify graphs with smallest eigenvalue beyond $-2$? In 1992, Bussemaker and Neumaier~\cite[Theorem~2.5]{BN92} identified the smallest $\la > 2$ for which $\G(\la)\setminus \G(2)$ contains precisely one connected graph, namely, the graph $E_{10}$ defined in \cref{fig:e2n}. More recently, motivated by a discrete-geometric question on spherical two-distance sets \cite{JTYZZ23}, Jiang and Polyanskii proved \cite[Theorem~2.10]{JP25} that the number of connected graphs in $\G(\la) \setminus \G(2)$ is finite for every $\la \in (2, \las)$. Here and throughout, \[\las := \rho^{1/2}+\rho^{-1/2} \approx 2.0198008871,\] where $\rho$ is the unique real root of $x^3 = x + 1$. Despite its algebraic definition, the peculiar constant $\las$ admits a natural interpretation within spectral graph theory.

\begin{proposition}[Hoffman~\cite{H72}] \label{lem:e2n}
    For every $n \in \N$ with $n \ge 4$, define the graph $E_n$ as in \cref{fig:e2n}. As $n \to \infty$, the largest eigenvalue of $E_n$ increases to $\las$, or equivalently, the smallest eigenvalue of $E_n$ decreases to $-\las$.\qed
\end{proposition}

\begin{figure}
    \centering
    \begin{tikzpicture}[very thick, scale=0.3, baseline=(v.base)]
        \coordinate (v) at (0,0);
        \draw (3,0) -- +(0,2) node[vertex]{};
        \draw (-1,0) node[vertex]{} -- (1,0) node[vertex]{} -- (3,0) node[vertex]{};
    \end{tikzpicture}\qquad
    \begin{tikzpicture}[very thick, scale=0.3, baseline=(v.base)]
        \coordinate (v) at (0,0);
        \draw (3,0) -- +(0,2) node[vertex]{};
        \draw (-1,0) node[vertex]{} -- (1,0) node[vertex]{} -- (3,0) node[vertex]{} --  (5,0) node[vertex]{};
    \end{tikzpicture}\qquad
    \begin{tikzpicture}[very thick, scale=0.3, baseline=(v.base)]
        \coordinate (v) at (0,0);
        \draw (3,0) -- +(0,2) node[vertex]{};
        \draw (-1,0) node[vertex]{} -- (1,0) node[vertex]{} -- (3,0) node[vertex]{} --  (5,0) node[vertex]{} -- (7,0) node[vertex]{};
    \end{tikzpicture}\qquad
    \begin{tikzpicture}[very thick, scale=0.3, baseline=(v.base)]
        \coordinate (v) at (0,0);
        \draw (3,0) -- +(0,2) node[vertex]{};
        \draw[dashed] (6.7,0) -- (11,0);
        \draw (-1,0) node[vertex]{} -- (1,0) node[vertex]{} -- (3,0) node[vertex]{} -- (5,0) node[vertex]{} -- (7,0) node[vertex]{};
        \draw (11,0) node[vertex]{} -- (13,0) node[vertex]{} -- (15,0) node[vertex]{};
        \draw [braket] (15,0.1) -- (-1,0.1) node [black,midway,yshift=-15pt] {\footnotesize $n-1$};
    \end{tikzpicture}
    \caption{$E_4$, $E_5$, $E_6$ and $E_n$.} \label{fig:e2n}
\end{figure}
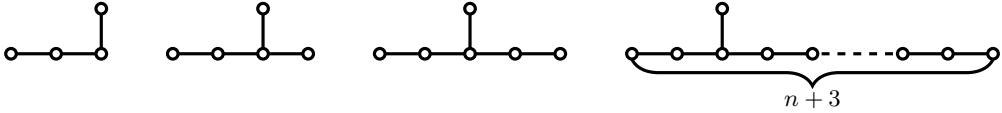

The finiteness result above, together with \cref{lem:e2n}, implies that $\las$ is the smallest $\la > 2$ for which $\G(\la)\setminus \G(2)$ contains infinitely many connected graphs. Naturally, the authors of \cite{JP25} posed the problem of determining all connected graphs in $\G(\las) \setminus \G(2)$.

We provide a complete classification of all connected graphs in $\G(\las) \setminus \G(2)$. Observe that the algebraic integer $-\las$ is not totally real and therefore cannot occur as a graph eigenvalue. Consequently, $\G(\las)\setminus \G(2)$ consists precisely of those graphs whose smallest eigenvalue lies in the open interval $(-\las, -2)$. To state our classification, we need the following definitions.

\begin{definition}[Augmented path extension]
    A \emph{rooted graph} $F_R$ is a graph $F$ equipped with a nonempty subset $R$ of vertices, which we refer to as \emph{roots} (depicted by solid circles). Given a rooted graph $F_R$ and $\ell \in \N$, the \emph{augmented path extension} $\ape{F_R, \ell}$ of the rooted graph $F_R$ is obtained from the disjoint union of $F$ and the rooted graph \etwo\ by adding a path $v_0 \dots v_\ell$ of length $\ell$, connecting $v_0$ to every vertex in $R$, and connecting $v_\ell$ to the two roots in \etwo. See \cref{fig:augmented} for a schematic drawing. A \emph{family of augmented path extensions} is the collection of graphs $\dset{\ape{F_R, \ell}}{\ell \ge \ell_0}$ for a fixed rooted graph $F_R$ and some $\ell_0 \in \N$.
\end{definition}

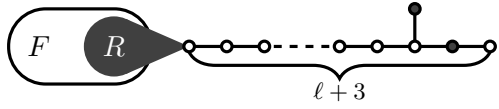
\begin{figure}
    \centering
    \begin{tikzpicture}[very thick, scale=0.5, baseline=(v.base)]
        \coordinate (v) at (0,0);
        \draw[rounded corners=14pt] (-7.8, -1) rectangle (-4, 1) {};
        \fill[darkgray] (-3,0) -- (-4.68,0.733) -- (-4.68,-0.733) -- cycle;
        \fill[darkgray] (-5, 0) circle (0.8);
        \draw[text=white] (-5,0) node{$R$};
        \draw (-7,0) node{$F$};
        \draw (3,1) node[root-vertex]{} -- (3,0);
        \draw[dashed] (-0.75,0) -- (1,0);
        \draw (-3,0) node[vertex]{} -- (-2,0) node[vertex]{} -- (-1,0) node[vertex]{};
        \draw (1,0) node[vertex]{} -- (2,0) node[vertex]{} -- (3,0) node[vertex]{} -- (4,0) node[root-vertex]{} -- (5,0) node[vertex]{};
        \draw [braket] (5,0.2) -- (-3,0.2) node [black,midway,yshift=-15pt] {\footnotesize $\ell + 3$};
    \end{tikzpicture}
    \caption[The augmented path extension.]{The augmented path extension $\ape{F_R, \ell}$.} \label{fig:augmented}
\end{figure}

\begin{definition}[Maverick graph]
    A \emph{maverick graph} is a connected graph with smallest eigenvalue in $(-\las,-2)$ that is not an augmented path extension of any rooted graph.
\end{definition}

We now present our main result in a nutshell before giving the detailed classification.

\begin{theorem}[Classification theorem at a glance] \label{thm:classification}
    The class $\G(\las) \setminus \G(2)$ of connected graphs with smallest eigenvalues in $(-\las, -2)$ consists precisely of $794$ families of augmented path extensions, and $4752$ maverick graphs with up to $19$ vertices.
\end{theorem}

In its weak form, our classification theorem says that every sufficiently large connected graph in $\G(\las)\setminus \G(2)$ resembles the graph $E_n$ for some $n$.

\begin{theorem} \label{thm:main1}
    There exists $N \in \N$ such that for every connected graph $G$ on more than $N$ vertices, if the smallest eigenvalue of $G$ is in $(-\las, -2)$, then $G$ is isomorphic to an augmented path extension of a rooted graph.
\end{theorem}

We establish \cref{thm:main1} in \cref{sec:forb} using tools developed for forbidden subgraphs characterization in \cite{JP25}, together with two classical results about generalized line graphs \cite{CDS80,CDS81,H77b,RSV81}. Since we rarely work with subgraphs that are not induced, we emphasize that \emph{all subgraphs are induced} throughout this paper.

To obtain the full classification, we proceed in two steps. First, we classify all augmented path extensions in $\G(\las)$ (excluding those already in $\G(2)$ is straightforward). Second, we enumerate the maverick graphs.

A key ingredient in pinning down the augmented path extensions in $\G(\las)$ is the following linear-algebraic lemma, which is proved in \cref{sec:lin-alg}. This lemma simplifies the task of determining whether an augmented path extension belongs in $\G(\las)$ to a finite computation.

\begin{lemma} \label{thm:reduction}
    For every rooted graph $F_R$ and $\ell \in \N$, the smallest eigenvalue of $\ape{F_R,\ell}$ is more than $-\las$ if and only if the same holds for $\ape{F_R, 0}$.
\end{lemma}

In \cref{sec:char}, we characterize all rooted graphs $F_R$ for which $\ape{F_R,0} \in \G(\las)$; these turn out to be precisely the line graphs of certain connected bipartite graphs $H_r$. We then enumerate the corresponding $H_r$ in \cref{sec:enum}. In retrospect, directly enumerating the rooted graphs $F_R$, which may have up to $14$ vertices and $6$ roots, would have been unwieldy without first identifying them as line graphs of bipartite graphs.

In \cref{sec:maverick}, we enumerate all maverick graphs. This is done by first generating graphs in $\G(\las)$ and then systematically removing those belonging to $\G(2)$ or arising as augmented path extensions. To exclude graphs in $\G(2)$, we apply a forbidden subgraph characterization of $\G(2)$ due to Kumar, Rao and Singhi \cite{KRS82}. To further eliminate augmented path extensions from $\G(\las)$, we use the observation that, roughly speaking, if $G \in \G(\las)$ is not an augmented path extension, then there exists a subgraph of $G$ with at most $11$ vertices that is not an augmented path extension.

It is natural to suspect that most graphs in $\G(\las) \setminus \G(2)$ arise as perturbations of graphs in $\G(2)$. In \cref{sec:twisted}, we show that every connected graph in $\G(\las) \setminus \G(2)$ with moderately large order can be obtained from a graph in $\G(2)$ by adding a pendant edge. Moreover, the underlying graph in $\G(2)$ turns out to be a line graph of a bipartite graph, and thus can be represented by a root system $A_n$ for some $n$.

\begin{corollary} \label{cor:main}
    For every connected graph $G$ on at least 18 vertices, if the smallest eigenvalue of $G$ is in $(-\las, -2)$, then there exists a unique leaf $v$ of $G$ such that $G-v$ is the line graph of a bipartite graph.
\end{corollary}

In \cref{sec:beyond}, we explore graphs that are slightly beyond $\G(\las)$, and we end the paper with open problems in \cref{sec:remarks}.

Parts of our proofs are computer-assisted with validated numerics. We explain our computer-aided proofs and how anyone can reproduce them independently. To help readers identify computer-assisted proofs, we employ a bespoke symbol \computer\ at the conclusion of each such proof. All our code is available as ancillary files in the arXiv version of this paper.

We deliberately craft our code without relying on third-party libraries, ensuring it can be adapted across different programming languages. This approach also offers the significant advantage that individuals interested in verifying our algorithms have the flexibility to use large language models to translate our code written in Ruby into their programming language of choice, providing a valuable starting point. It is important to note that while this translation serves as an efficient initial step, further human refinement may be necessary to ensure the code fully captures the nuances of our original implementations.

\section{Forbidden subgraph characterization} \label{sec:forb}

The proof of \cref{thm:main1} centers around the notion of \emph{generalized line graphs} originally defined by Hoffman~\cite{H69}. Although we do not need their definition, for concreteness we state the alternative definition from \cite{CDS81}.

\begin{definition}[Graph with petals and generalized line graph]
    A \emph{graph with petals} is a multigraph $\hat{H}$ obtained from a graph by adding pendant double edges. A \emph{generalized line graph} $L(\hat{H})$ is the line graph of a graph $\hat{H}$ with petals where two vertices of $L(\hat{H})$ are adjacent if and only if the corresponding edges in $\hat{H}$ have exactly one vertex in common. See \cref{fig:generalized-line-graph} for a schematic drawing.
\end{definition}

\begin{figure}
    \centering
    \begin{tikzpicture}[very thick, scale=0.5, baseline=(v.base)]
        \coordinate (v) at (0,0);
        \draw (-2,0) arc[start angle=30, end angle=90, radius=1];
        \draw (-2,0) arc[start angle=-90, end angle=-150, radius=1] node[vertex]{};
        \draw (-2,0) arc[start angle=90, end angle=150, radius=1];
        \draw (-2,0) arc[start angle=-30, end angle=-90, radius=1] node[vertex]{};
        \draw (0,1) arc[start angle=-30, end angle=30, radius=1];
        \draw (0,1) arc[start angle=-150, end angle=-210, radius=1] node[vertex]{};
        \draw (2,0) arc[start angle=180, end angle=120, radius=1];
        \draw (2,0) arc[start angle=300, end angle=360, radius=1] node[vertex]{};
        \draw (2,0) arc[start angle=60, end angle=0, radius=1];
        \draw (2,0) arc[start angle=180, end angle=240, radius=1] node[vertex]{};
        \draw (2,0) arc[start angle=120, end angle=60, radius=1];
        \draw (2,0) arc[start angle=-120, end angle=-60, radius=1] node[vertex]{};
        \draw (-2,0) -- (0,-1) node[vertex]{} -- (2,0);
        \draw (-2,0) node[vertex]{} -- (2,0) node[vertex]{} -- (0,1) node[vertex]{} -- cycle;
    \end{tikzpicture}\qquad\qquad
    \begin{tikzpicture}[very thick, scale=0.5, baseline=(v.base)]
        \coordinate (v) at (0,0);
        \draw[fill=litegray, draw=none] (0,0) -- (-5.5,0.2) -- (-5.5,-0.2) -- cycle;
        \draw[fill=litegray, draw=none] (-3,1.5) -- (-5.5,0.25) -- (-5.5,-0.3) -- cycle;
        \draw[fill=litegray, draw=none] (-3,-1.5) -- (-5.5,0.3) -- (-5.5,-0.25) -- cycle;
        \draw[fill=litegray, draw=none] (-3,1.5) -- (0,3.3) -- (0,2.8) -- cycle;
        \draw[fill=litegray, draw=none] (3,1.5) -- (0,3.3) -- (0,2.8) -- cycle;
        \draw[fill=litegray, draw=none] (0,0) -- (6,0.22) -- (6,-0.22) -- cycle;
        \draw[fill=litegray, draw=none] (3,1.5) -- (6,0.27) -- (6,-0.32) -- cycle;
        \draw[fill=litegray, draw=none] (3,-1.5) -- (6,0.32) -- (6,-0.27) -- cycle;
        \draw (-3,1.5) node[vertex]{} -- (3,1.5) -- (3,-1.5) node[vertex]{} -- (-3,-1.5) -- (-3,1.5) -- (3,-1.5);
        \draw (-3,-1.5) node[vertex]{} -- (0,0) node[vertex]{} -- (3,1.5)node[vertex]{};
        \draw[fill=litegray, draw=none] (-5.5, 0) circle (1.1);
        \draw (-5, -0.5) node[vertex]{} -- (-5,0.5) node[vertex]{}-- (-6,0.5) node[vertex]{} -- (-6,-0.5) node[vertex]{} -- cycle;
        \draw[fill=litegray, draw=none] (0, 3) circle (0.9);
        \draw (-0.5,3) node[vertex]{};
        \draw (0.5,3) node[vertex]{};
        \draw[fill=litegray, draw=none] (6, 0) circle (1.4);
        \draw (7,0) -- ++(150:{sqrt(3)}) -- ++(0,{-sqrt(3)}) -- cycle;
        \draw (5,0) -- ++(30:{sqrt(3)}) -- ++(0,{-sqrt(3)}) -- cycle;
        \draw (7,0) node[vertex]{} -- ++(120:1) node[vertex]{} -- ++(-1,0) node[vertex]{} -- ++(240:1) node[vertex]{} -- ++(-60:1) node[vertex]{} -- ++(1,0) node[vertex]{} -- cycle;
    \end{tikzpicture}
    \caption{A graph $\hat{H}$ with petals and a schematic drawing of its line graph $L(\hat{H})$.} \label{fig:generalized-line-graph}
\end{figure}
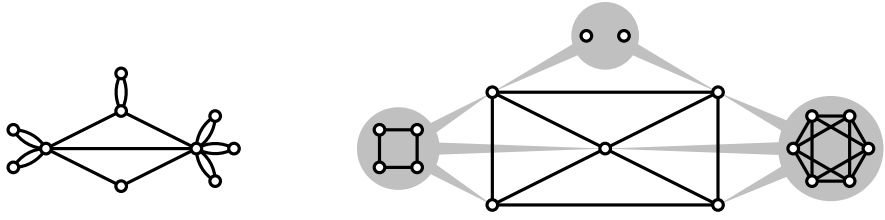

We need the following properties --- just like line graphs, all the generalized line graphs are in $\G(2)$, and they have a finite forbidden subgraph characterization.

\begin{theorem}[Theorem~2.1 of Hoffman~\cite{H77b}] \label{thm:generalized-line-smallest-ev}
    The smallest eigenvalue of a generalized line graph is at least $-2$. \qed
\end{theorem}

\begin{theorem}[Cvetkovi\'{c}, Doob, and Simi\'{c}~\cite{CDS80,CDS81}, and Rao, Singhi, and Vijayan~\cite{RSV81}] \label{thm:forb-char-generalized-line}
    There are 31 minimal forbidden subgraphs, one of which is $E_6$ defined in \cref{fig:e2n}, for the family $\D_\infty$ of generalized line graphs. \qed
\end{theorem}

To prove \cref{thm:main1}, the strategy is to forbid specific subgraphs, including all the minimal forbidden subgraphs for $\D_\infty$ except $E_6$, in every sufficiently large connected graph in $\G(\las)\setminus\G(2)$. To that end, we recall from \cite{JP25} the following ways to extend a rooted graph.

\begin{definition}[Path extension, clique extension, and path-clique extension] \label{def:extension}
    Given a rooted graph $F_R$, $\ell \in \N$, and $m \in \N^+$,
    \begin{enumerate}[label=(\alph*)]
        \item the \emph{path extension} $(F_R, \ell)$ is obtained from $F$ by adding a path $v_0 \dots v_\ell$ of length $\ell$, and connecting $v_0$ to every vertex in $R$.\footnote{The path of length $0$ is simply a single vertex.}
        \item the \emph{path-clique extension} $(F_R, \ell, K_m)$ is further obtained from $(F_R, \ell)$ by adding a clique of order $m$, and connecting every vertex in the clique to $v_\ell$.
        \item the \emph{clique extension} $(F_R, K_m)$ is obtained from $F$ by adding a clique of order $m$, and connecting every vertex in the clique to every vertex in $R$.
    \end{enumerate}
\end{definition}

\cref{fig:extensions} consists of schematic drawings of the path extension $(F_R, \ell)$, the path-clique extension $(F_R, \ell, K_m)$, and the clique extension $(F_R, K_m)$.

\begin{figure}[t]
    \centering
    \begin{tikzpicture}[very thick, scale=0.5, baseline=(v.base)]
        \coordinate (v) at (0,0);
        \draw[rounded corners=14pt] (-7.8, -1) rectangle (-4, 1) {};
        \fill[darkgray] (-3,0) -- (-4.68,0.733) -- (-4.68,-0.733) -- cycle;
        \fill[darkgray] (-5, 0) circle (0.8);
        \draw[white] (-5,0) node{$R$};
        \draw (-7,0) node{$F$};
        \draw[dashed] (-2.05,0) -- (-.2,0);
        \draw (-3,0) node[vertex]{} -- (-2,0) node[vertex]{};
        \draw (-.25,0) node[vertex]{} -- (.75,0) node[vertex]{};
        \draw [braket] (.75,0.2) -- (-3,0.2) node [black,midway,yshift=-15pt] {\footnotesize $\ell + 1$};
    \end{tikzpicture}\qquad%
    \begin{tikzpicture}[very thick, scale=0.5, baseline=(v.base)]
        \coordinate (v) at (0,0);
        \draw[rounded corners=14pt] (-7.8, -1) rectangle (-4, 1) {};
        \fill[darkgray] (-3,0) -- (-4.68,0.733) -- (-4.68,-0.733) -- cycle;
        \fill[darkgray] (0.75,0) -- (2.03,0.96) -- (2.03,-0.96) -- cycle;
        \fill[darkgray] (-5, 0) circle (0.8);
        \fill[darkgray] (2.75, 0) circle (1.2);
        \draw[white] (-5,0) node{$R$};
        \draw (-7,0) node{$F$};
        \draw[white] (2.75,0) node{$K_m$};
        \draw[dashed] (-2.05,0) -- (-.2,0);
        \draw (-3,0) node[vertex]{} -- (-2,0) node[vertex]{};
        \draw (-.25,0) node[vertex]{} -- (.75,0) node[vertex]{};
        \draw [braket] (.75,0.2) -- (-3,0.2) node [black,midway,yshift=-15pt] {\footnotesize $\ell + 1$};
    \end{tikzpicture}\qquad%
    \begin{tikzpicture}[very thick, scale=0.5, baseline=(v.base)]
        \coordinate (v) at (0,0);
        \draw[rounded corners=14pt] (-7.8, -1) rectangle (-4, 1) {};
        \fill[darkgray] (-5.107,0.79) -- (-2.16,1.19) -- (-2.16,-1.19) -- (-5.107,-0.79) -- cycle;
        \fill[darkgray] (-5, 0) circle (0.8);
        \fill[darkgray] (-2, 0) circle (1.2);
        \draw[white] (-5,0) node{$R$};
        \draw (-7,0) node{$F$};
        \draw[white] (-2,0) node{$K_m$};
    \end{tikzpicture}
    \caption{Three extensions of a rooted graph $F_R$.} \label{fig:extensions}
\end{figure}

To forbid a subgraph $F$ in every sufficiently large connected graph in $\G(\la)$, it is necessary that no sufficiently large path extension of $F$ belongs to $\G(\la)$. This condition turns out to be sufficient when $\la \ge 2$. Hereafter, we denote the smallest eigenvalue of a graph $G$ by $\la_1(G)$.

\begin{lemma} \label{lem:forb-extension}
    Suppose that $F$ is a graph and $\la \ge 2$. If $\lim_{\ell \to \infty} \la_1(F_R, \ell) < -\la$ for every nonempty vertex subset $R$ of $F$, then there exists $N \in \N$ such that $F$ is never a subgraph of any connected graph on more than $N$ vertices with smallest eigenvalue at least $-\la$.
\end{lemma}

For the proof of \cref{lem:forb-extension}, we need the following three results. The first two results automatically extend the inequality condition on path extensions to path-clique extensions and clique extensions.

\begin{lemma}[Lemma~2.15 of Jiang and Polyanskii~\cite{JP25}] \label{lem:path-clique-extension}
    For every rooted graph $F_R$ and $\la \ge 2$, if $\lim_{\ell\to\infty} \la_1(F_R, \ell) < -\la$, then there exists $m \in \N^+$ such that $\la_1(F_R, \ell, K_m) < -\la$ for every $\ell \in \N$. \qed
\end{lemma}

\begin{lemma} \label{lem:clique-extension}
    For every rooted graph $F_R$ and $\la \ge 2$, if $\lim_{\ell\to\infty} \la_1(F_R, \ell) < -\la$, then there exists $m \in \N^+$ such that $\la_1(F_R, K_m) < -\la$.
\end{lemma}
\begin{proof}
    Pick $\ell \in \N^+$ such that $\la_1(F_R, \ell) < -\la$. Set $\la' := -\la_1(F_R, \ell)$. Let $v_0\dots v_\ell$ be the path added to $F$ to obtain $(F_R, \ell)$, where the vertex $v_0$ is connected to every vertex in $R$, and let $\bx\colon V(F) \cup \sset{v_0, \dots, v_\ell} \to \R$ be an eigenvector of $(F_R, \ell)$ associated with $-\la'$. We abuse notation and write $x_i$ in place of $x_{v_i}$ for $i \in \sset{0,\dots, \ell}$. Define $\tbx \colon V(F) \cup V(K_m) \to \R$ by
    \[
        \tilde{x}_v = \begin{cases}
            x_v & \text{if }v \in V(F); \\
            x_0/m  & \text{if }v \in V(K_m).
        \end{cases}
    \]

    We claim that $\sum_{v \in V(F)}x_v^2 > 0$. Indeed, assume for the sake of contradiction that $x_v = 0$ for $v \in V(F)$. Using $-\la' x_i = \sum_{u\sim v_i}x_u$ for $i \in \sset{0, \dots, \ell}$, where the sum is taken over all vertices $u$ that are adjacent to $v_i$ in $(F_R, \ell)$, we obtain that \[
        (A_P + \la' I)\begin{pmatrix}
            x_0 \\
            \vdots \\
            x_\ell
        \end{pmatrix} = \bm{0},
    \]
    where $P$ denotes the path $v_0\dots v_\ell$. Since $\la_1(P) > -2 > -\la'$, the matrix $A_P + \la' I$ is positive definite, contradicting the assumption that $\bx$ is nonzero.

    Because $\sum_{v \in V(F)}x_v^2 > 0$, clearly $\tbx$ is a nonzero vector. We compute
    \[
        \tbx^\T \tbx = \sum_{v \in V(F)}x_v^2 + m(x_0/m)^2.
    \]
    Moreover we can compute $\tbx^\T A_{(F_R, K_m)}\tbx$ as follows
    \[
        \tbx^\T A_{(F_R, K_m)}\tbx = \sum_{u,v \in V(F_R, 0)\colon u\sim v}x_ux_v + m(m-1)(x_0/m)^2.
    \]
    Since $\bx$ is an eigenvector of $(F_R, \ell)$ associated with $-\la'$, we obtain that
    \[
        \sum_{u, v \in V(F_R, 0)\colon u\sim v} x_ux_v + x_0x_1 = \sum_{v \in V(F_R, 0)}x_v \sum_{u\sim v}x_u = -\la'\sum_{v \in V(F)}x_v^2 - \la' x_0^2.
    \]
    Thus $\tbx^\T A_{(F_R, K_m)}\tbx$ can be simplified to
    \[
        \tbx^\T A_{(F_R, K_m)}\tbx = -\la'\sum_{v \in V(F)}x_v^2 -\la'x_0^2 - x_0x_1 + m(m-1)(x_0/m)^2.
    \]
    The Rayleigh principle says that $\la_1(F_R, K_m)$ is at most
    \[
        \frac{-\la'\sum_{v \in V(F)}x_v^2 - \la'x_0^2 - x_0x_1 + m(m-1)(x_0/m)^2}{\sum_{v \in V(F)}x_v^2 + m(x_0/m)^2},
    \]
    which, as $m\to\infty$, approaches
    \[
        -\la' - \frac{(\la'-2)x_0^2+x_0(x_0 + x_1)}{\sum_{v \in V(F)}x_v^2}.
    \]
    Here we used the claim that the denominator in the limit is positive.

    Recall that $\la' = -\la_1(F_R, \ell) > \la \ge 2$. It suffices to show that $x_0(x_0 + x_1)\ge 0$. In fact, we prove inductively that $x_i(x_i + x_{i+1}) \ge 0$ for $i \in \sset{\ell-1, \dots, 0}$. The base case where $i = \ell-1$ follows immediately from $-\la' x_\ell = x_{\ell-1}$ and $\la' > 2$. For the inductive step, using $-\la' x_{i+1} = x_i + x_{i+2}$ and $\la' > 2$, we obtain
    \begin{multline*}
        x_{i+1}(x_{i+1} + x_{i+2}) = x_{i+1}(x_{i+1} - \la'x_{i+1}-x_i) = -(\la'-2)x_{i+1}^2 - (x_i + x_{i+1})x_{i+1} \\
        \le -(x_i + x_{i+1})x_{i+1} = -(x_i + x_{i+1})^2 + x_i(x_i + x_{i+1}) \le x_i(x_i + x_{i+1}),
    \end{multline*}
    which implies that $x_i(x_i+x_{i+1}) \ge 0$ by the inductive hypothesis.
\end{proof}

The third result shows that forbidding a star and an extension family of $F$ effectively forbids $F$ itself in every sufficiently large connected graph. Denote by $S_k$ the star on $k+1$ vertices.

\begin{definition}[Extension family]
    Given a graph $F$ and $\ell, m \in \N^+$, the \emph{extension family} $\X(F, \ell, m)$ of $F$ consists of the path-extension $(F_R, \ell)$, the path-clique extension $(F_R, \ell_0, K_m)$, and the clique extension $(F_R, K_m)$, where $R$ ranges over the nonempty vertex subsets of $F$, and $\ell_0$ ranges over $\sset{0, \dots, \ell-1}$.
\end{definition}

\begin{lemma}[Lemma~2.6 of Jiang and Polyanskii~\cite{JP25}] \label{lem:forbid-extensions}
    For every graph $F$ and $k, \ell, m \in \N^+$, there exists $N \in \N$ such that for every connected graph $G$ on more than $N$ vertices, if no member in $\sset{S_k} \cup \X(F, \ell, m)$ is a subgraph of $G$, then neither is $F$.\qed
\end{lemma}

\begin{remark}
    Ding, Oporowski, Oxley, and Vertigan proved that for every $n \in \N$, every sufficiently large connected graph contains $S_n$, $P_n$, or $K_n$ as a subgraph \cite[Theorem 5.3]{DOOV}. \cref{lem:forbid-extensions} extends their result.
\end{remark}

We now have all of the ingredients needed to establish \cref{lem:forb-extension}.

\begin{proof}[Proof of \cref{lem:forb-extension}]
    Choose $\ell \in \N^+$ such that $\la_1(F_R, \ell) < -\la$ for every nonempty $R \subseteq V(F)$. According to \cref{lem:path-clique-extension,lem:clique-extension}, choose $m \in \N^+$ such that $\la_1(F_R, \ell_0, K_m) < -\la$ for every nonempty $R \subseteq V(F)$ and every $\ell_0 \in \N$, and $\la_1(F_R, K_m) < -\la$ for every nonempty $R \subseteq V(F)$.
    
    Suppose that $G$ is a graph with $\la_1(G) > -\la$. The choice of $\ell$ and $m$ ensures that no member in $\X(F, \ell, m)$ is a subgraph of $G$. Furthermore, the star $S_k$, whose smallest eigenvalue is $-\sqrt{k}$, cannot be a subgraph of $G$ for any $k \in \N^+$ satisfying $\sqrt{k} > \la$. Finally we apply \cref{lem:forbid-extensions} to obtain the desired $N \in \N$.
\end{proof}

With \cref{lem:forb-extension} at our disposal, we return to forbidding certain subgraphs in every sufficiently large connected graph in $\G(\las)\setminus \G(2)$. The following computation ensures that $E_6$ occurs as a subgraph in every sufficiently large connected graph in $\G(\las) \setminus \G(2)$.

\begin{lemma}[Lemma~2.14 of Jiang and Polyanskii \cite{JP25}] \label{lem:computation}
    For every minimal forbidden subgraph $F$ for the family $\D_\infty$ of generalized line graphs, if $F$ is not isomorphic to $E_6$, then $\lim_{\ell \to \infty} \la_1(F_R, \ell) < -\las$ for every nonempty vertex subset $R$ of $F$. \qed
\end{lemma}

In addition, we further forbid small supergraphs of $E_6$ through the following computation, proved with computer assistance in \cref{sec:app}.

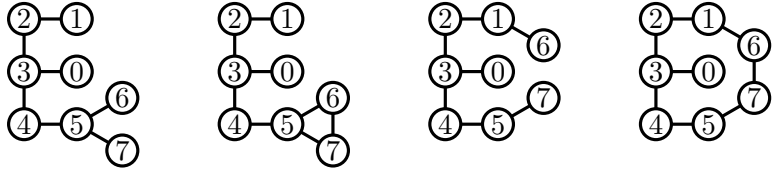
\begin{figure}
    \centering
    \begin{tikzpicture}[very thick, scale=0.7]
        \begin{scope}
            \draw (0,0) -- (1,0);
            \draw (1,1) -- (0,1) -- (0,-1) -- (1,-1);
            \draw (1,-1) -- ++(30:1) node[circ]{6};
            \draw (1,-1) -- ++(-30:1) node[circ]{7};
            \node[circ] at (1,0) {0};
            \node[circ] at (1,1) {1};
            \node[circ] at (0,1) {2};
            \node[circ] at (0,0) {3};
            \node[circ] at (0,-1) {4};
            \node[circ] at (1,-1) {5};
        \end{scope}
        \begin{scope}[shift={(4,0)}]
            \draw (0,0) -- (1,0);
            \draw (1,1) -- (0,1) -- (0,-1) -- (1,-1);
            \draw (1,-1) -- ++(30:1) node[circ]{6} -- ++(0,-1);
            \draw (1,-1) -- ++(-30:1) node[circ]{7};
            \node[circ] at (1,0) {0};
            \node[circ] at (1,1) {1};
            \node[circ] at (0,1) {2};
            \node[circ] at (0,0) {3};
            \node[circ] at (0,-1) {4};
            \node[circ] at (1,-1) {5};
        \end{scope}
        \begin{scope}[shift={(8,0)}]
            \draw (0,0) -- (1,0);
            \draw (1,1) -- (0,1) -- (0,-1) -- (1,-1);
            \draw (1,1) -- ++(-30:1) node[circ]{6};
            \draw (1,-1) -- ++(30:1) node[circ]{7};
            \node[circ] at (1,0) {0};
            \node[circ] at (1,1) {1};
            \node[circ] at (0,1) {2};
            \node[circ] at (0,0) {3};
            \node[circ] at (0,-1) {4};
            \node[circ] at (1,-1) {5};
        \end{scope}
        \begin{scope}[shift={(12,0)}]
            \draw (0,0) -- (1,0);
            \draw (1,1) -- (0,1) -- (0,-1) -- (1,-1);
            \draw (1,1) -- ++(-30:1) node[circ]{6} -- ++(0,-1);
            \draw (1,-1) -- ++(30:1) node[circ]{7};
            \node[circ] at (1,0) {0};
            \node[circ] at (1,1) {1};
            \node[circ] at (0,1) {2};
            \node[circ] at (0,0) {3};
            \node[circ] at (0,-1) {4};
            \node[circ] at (1,-1) {5};
        \end{scope}
        \end{tikzpicture}
    \caption{Four $8$-vertex graphs $F$ with two vertices $v_6$ and $v_7$ such that $F - \sset{v_6, v_7}$ is isomorphic to $E_6$, and both $F - v_6$ and $F - v_7$ are isomorphic to $E_7$.} \label{fig:e-graphs}
\end{figure}

\begin{proposition} \label{lem:seven-eight-vertex}
    For every connected graph $F$ that contains $E_6$ as a subgraph, if $F$ is a $7$-vertex graph that is not isomorphic to $E_7$, or $F$ is an $8$-vertex graph in \cref{fig:e-graphs}, then $\lim_{\ell \to \infty} \la_1(F_R, \ell) < -95/47$ for every nonempty vertex subset $R$ of $F$.
\end{proposition}

\begin{proof}[Proof of \cref{thm:main1}] \label{proof:main1}
    By combining \cref{lem:forb-extension,lem:computation}, we obtain $N_1 \in \N$ such that for every minimal forbidden subgraph $F$ for $\D_\infty$, if $F$ is not isomorphic to $E_6$, then $F$ is not a subgraph of any connected graph on more than $N_1$ vertices with smallest eigenvalue more than $-\las$. Combining \cref{lem:forb-extension,lem:seven-eight-vertex}, we obtain $N_2 \in \N$ such that for every connected graph $F$ that contains $E_6$ as a subgraph, if $F$ is a $7$-vertex graph that is not isomorphic to $E_7$, or $F$ is an $8$-vertex graph in \cref{fig:e-graphs}, then $F$ is not a subgraph of any connected graph on more than $N_2$ vertices with smallest eigenvalue more than $-\las$. Here, we use the fact that $\las \approx 2.01980 < 2.02127 \approx 95/47$.

    Suppose that $G$ is a connected graph on more than $N := \max(N_1, N_2)$ vertices with smallest eigenvalue in $(-\las,-2)$. Observe from \cref{thm:generalized-line-smallest-ev} that $G$ is not a generalized line graph. Thus $G$ contains a subgraph $E$ that is a minimal forbidden subgraph for $\D_\infty$. The choice of $N_1$ forces $E$ to be isomorphic to $E_6$. The choice of $N_2$ ensures that $G[V(E) \cup \sset{v}]$ is isomorphic to $E_7$ for every vertex $v$ at distance $1$ from $E$ in $G$, and, in fact, every such vertex $v$ is adjacent to the same vertex of $E$. In particular, $G$ is an augmented path extension of a rooted graph.
\end{proof}

\section{The linear-algebraic lemma} \label{sec:lin-alg}

We generalize \cref{thm:reduction} to a class of graphs that encapsulates both augmented path extensions and path-clique extensions.

\begin{definition}[Path augmentation]
    Given two rooted graphs $F_R$ and $G_S$ and $\ell \in \N$, the \emph{path augmentation} $(F_R, \ell, G_S)$ of $F_R$ and $G_S$ is obtained from the disjoint union of $F$ and $G$ by adding a path $v_0\dots v_\ell$ of length $\ell$, connecting $v_0$ to every vertex in $R$, and connecting $v_\ell$ to every vertex in $S$. See \cref{fig:fr-l-gs} for a schematic drawing.
\end{definition}

\begin{figure}
    \centering
    \begin{tikzpicture}[very thick, scale=0.5, baseline=(v.base)]
        \coordinate (v) at (0,0);
        \draw[rounded corners=14pt] (-7.8, -1) rectangle (-4, 1) {};
        \draw[rounded corners=14pt] (7.8, -1) rectangle (4, 1) {};
        \fill[darkgray] (-3,0) -- (-4.68,0.733) -- (-4.68,-0.733) -- cycle;
        \fill[darkgray] (3,0) -- (4.68,0.733) -- (4.68,-0.733) -- cycle;
        \fill[darkgray] (-5, 0) circle (0.8);
        \fill[darkgray] (5, 0) circle (0.8);
        \draw[text=white] (-5,0) node{$R$};
        \draw[text=white] (5,0) node{$S$};
        \draw (-7,0) node{$F$};
        \draw (7,0) node{$G$};
        \draw[dashed] (-0.75,0) -- (1,0);
        \draw (-3,0) node[vertex]{} -- (-2,0) node[vertex]{} -- (-1,0) node[vertex]{};
        \draw (1,0) node[vertex]{} -- (2,0) node[vertex]{} -- (3,0) node[vertex]{};
        \draw [braket] (3,0.2) -- (-3,0.2) node [black,midway,yshift=-15pt] {\footnotesize $\ell + 1$};
    \end{tikzpicture}
    \caption{The path augmentation $(F_R, \ell, G_S)$.} \label{fig:fr-l-gs}
\end{figure}
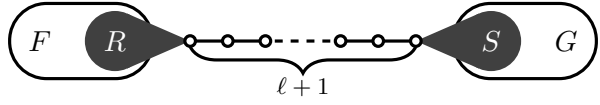

\begin{lemma} \label{lem:reduction-gen}
    Suppose that $G_S$ is a rooted graph. If $\la = -\lim_{\ell \to \infty} \la_1(G_S,\ell)$, and $\la_1(G) > -\la$, then the following holds. For every rooted graph $F_R$ and $\ell \in \N$, the smallest eigenvalue of $(F_R, \ell, G_S)$ is more than $-\la$ if and only if the same holds for $(F_R, 0, G_S)$.
\end{lemma}

\begin{proof}[Proof of \cref{thm:reduction}]
    Take $G_S = \etwo$ in \cref{lem:reduction-gen}, and observe that $(\etwo, \ell)$ is just $E_{\ell+4}$. \cref{thm:reduction} follows immediately from \cref{lem:e2n}.
\end{proof}

To provide additional context for \cref{lem:reduction-gen}, we discuss the behavior of the smallest eigenvalue of the path augmentation $(F_R, \ell, G_S)$ as $\ell \to \infty$. The following result follows from \cite[Proposition~3.5]{BB24}. We provide a self-contained proof here.

\begin{lemma} \label{lem:limit}
    For all rooted graphs $F_R$ and $G_S$, \[\lim_{\ell \to \infty}\la_1(F_R, \ell, G_S) = \min\left(\lim_{\ell \to \infty}\la_1(F_R, \ell), \lim_{\ell \to \infty}\la_1(G_S, \ell)\right).\]
\end{lemma}

\begin{proof}
    Set $\la_{F_R} = \lim_{\ell\to\infty} \la_1(F_R, \ell)$ and $\la_{G_S} = \lim_{\ell\to\infty} \la_1(G_S, \ell)$. Since both $(F_R, \ell)$ and $(G_S, \ell)$ are subgraphs of $(F_R, \ell, G_S)$, clearly $\limsup_{\ell \to\infty}\la_1(F_R, \ell, G_S) \le \min(\la_{F_R},\la_{G_S})$.
    
    To see the reverse, let $v_0\dots v_\ell$ be the path added to the disjoint union of $F$ and $G$ to obtain $(F_R, \ell, G_S)$, and let $\bx \colon V(F_R, \ell, G_S) \to \R$ be a unit eigenvector associated with $\la_1(F_R, \ell, G_S)$. We abuse notation and write $x_i$ in place of $x_{v_i}$. Choose $k \in \sset{0,\dots,\ell-1}$ such that $x_kx_{k+1}$ reaches the minimum in absolute value. In particular, using the inequality $\abs{x_ix_{i+1}} \le (x_i^2+x_{i+1}^2)/2$, we obtain
    \[
        \abs{x_kx_{k+1}} \le \frac{1}{\ell} \sum_{i=0}^{\ell-1}\abs{x_ix_{i+1}} \le \frac{1}{\ell}\sum_{i=0}^{\ell}x_i^2 \le \frac{1}{\ell}.
    \]
    Notice that removing the edge $v_kv_{k+1}$ disconnects $(F_R, \ell, G_S)$ into subgraphs $F' := (F_R, k)$ and $G' := (G_S, \ell-k-1)$. Let $\bx_1$ and $\bx_2$ be the restrictions of $\bx$ to $V(F')$ and $V(G')$ respectively. Finally, we bound the smallest eigenvalue of $(F_R, \ell, G_S)$ as follows:
    \begin{multline*}
        \la_1(F_R, \ell, G_S) = \bx^\T A_{(F_R, \ell, G_S)} \bx 
        = \bx_1^\T A_{F'} \bx_1 + 2x_kx_{k+1} + \\ + \bx_2^\T A_{G'}\bx_2
        \ge \la_{F_R} \bx_1^\T\bx_1 - 2/\ell + \la_{G_S}\bx_2^\T \bx_2
        \ge \min(\la_{F_R}, \la_{G_S}) - 2/\ell,
    \end{multline*}
    which implies that $\liminf_{\ell\to\infty} \la_1(F_R, \ell, G_S) \ge \min(\la_{F_R}, \la_{G_S})$.
\end{proof}

\begin{example}
    Consider the two cases where $F_R \in \sset{~\ktwoc~, ~\ktwo~}$ and $G_S = \etwo$. In both cases, because $\lim_{\ell \to \infty}\la_1(F_R,\ell) = -2$ and $\lim_{\ell \to \infty}\la_1(G_S,\ell) = -\las$, according to \cref{lem:limit}, $\la_1(F_R, \ell, G_S)$ approaches $-\las$ as $\ell \to \infty$. Interestingly, the smallest eigenvalue approaches $-\las$ in different ways --- $\la_1(\ktwoc, \ell, \etwo)$ approaches $-\las$ from below, whereas $\la_1(\ktwo, \ell, \etwo)$ approaches from above. \cref{lem:reduction-gen} rules out other ways $\la_1(F_R, \ell, G_S)$ could approach its limit.
\end{example}

We devote the rest of the section to the proof of \cref{lem:reduction-gen}. Denote by $E_{v,v}$ the matrix unit where the $(v,v)$-entry with value $1$ is the only nonzero entry. We first characterize $\lim_{\ell \to \infty}\la_1(G_S, \ell)$.

\begin{lemma} \label{lem:lim-smallest-ev}
    Suppose that $G_S$ is a rooted graph. Let $v_0$ be the vertex in $V(G_S,0)\setminus V(G)$. If $\la = -\lim_{\ell\to\infty}\la_1(G_S,\ell)$, then the set of $x \ge 2$, for which the matrix
    \begin{equation} \label{eqn:matrix-gs}
        A_{(G_S,0)} + xI - \left(x/2 - \sqrt{x^2/4-1}\right)E_{v_0,v_0}
    \end{equation}
    is positive semidefinite, is equal to $[\la, \infty)$, and moreover, the above matrix is singular when $x = \la$.
\end{lemma}

\begin{proof}
    Denote by $P_\ell$ the path of length $\ell$. Since $\lim_{\ell \to \infty}\la_1(P_\ell) = -2$, clearly $\la \ge 2$. Let $x \ge 2$ be chosen later, and for every $n \in \N^+$, set $d_n := \det\mleft(A_{P_{n-1}} + xI\mright)$. Using Laplace expansion, we derive the linear recurrence $d_{n+2} = x d_{n+1} - d_n$ with the initial conditions $d_0 = 1$ and $d_1 = x$. It follows from the classical theory of linear recurrence that $\lim_{\ell \to \infty} d_{\ell-1}/d_\ell = x/2 - \sqrt{x^2/4-1}$.

    \begin{claim*}
        For every $\ell \in \N^+$, the matrix $A_{(G_S, \ell)} + x I$ is positive semidefinite if and only if $A_{(G_S, 0)} + xI - (d_{\ell-1}/d_\ell)E_{v_0,v_0}$ is positive semidefinite.
    \end{claim*}

    \begin{claimproof}[Proof of Claim]
        We partition the matrix $A_{(G_S,\ell)} + xI$ into the following blocks:
        \[
            \begin{pmatrix}
                A_{(G_S,0)} + xI & B \\
                B^\T & C
            \end{pmatrix}.
        \]
        Since $C = A_{P_{\ell-1}} + xI$, and $\la_1(P_{\ell-1}) > -2 \ge -x$, the block $C$ is positive definite. Therefore, the above block matrix is positive semidefinite if and only if the Schur complement $A_{(G_S, 0)} + xI - BC^{-1}B^\T$ of $C$ is positive semidefinite. Let $v_1$ be the vertex in $V(G_S, 1) \setminus V(G_S, 0)$. Since the only nonzero entry of $B$ is its $(v_0, v_1)$ entry, the matrix $BC^{-1}B^\T$ simplifies to $(C^{-1})_{v_1, v_1}E_{v_0,v_0}$. Cramer's rule yields $(C^{-1})_{v_1,v_1}=\det C' / \det C$, where $C'$ is obtained from $C$ by removing the $v_1$-th row and column. To finish the proof of the claim, note that $\det C' = d_{\ell-1}$ and $\det C = d_\ell$.
    \end{claimproof}

    First, we consider the case where $x \ge \la$. The claim implies that $A_{(G_S,0)} + xI - (d_{\ell-1}/d_\ell)E_{v_0,v_0}$ is positive semidefinite. Sending $\ell$ to $\infty$, we know that the matrix in \eqref{eqn:matrix-gs} is positive semidefinite when $x \ge \la$. Next, we consider the case where $x \in [2, \la)$. The claim implies that $A_{(G_S,0)} + xI - (d_{\ell-1}/d_\ell)E_{v_0,v_0}$ is not positive semidefinite for sufficiently large $\ell$. Sending $\ell$ to $\infty$, we know that the matrix in \eqref{eqn:matrix-gs} is not positive semidefinite when $x \in [2, \la)$. Finally, assume for the sake of contradiction that the matrix in \eqref{eqn:matrix-gs} is positive definite when $x = \la$. We can then decrease $x$ slightly so that the matrix in \eqref{eqn:matrix-gs} is still positive definite, which yields a contradiction.
\end{proof}

We adopt the convention that the path extension $(G_S, -1)$ is just the graph $G$.

\begin{corollary} \label{cor:ga}
    Suppose that $G_S$ is a rooted graph. If $\la = -\lim_{\ell \to \infty} \la_1(G_S, \ell)$, then for every $\ell \in \N$, the matrix
    \[
        A_{(G_S, \ell)} + \la I - \left(\la/2 - \sqrt{\la^2/4-1}\right) E_{v_\ell,v_\ell}
    \]
    is singular, where $v_\ell$ is the vertex in $V(G_S, \ell)\setminus V(G_S, \ell-1)$.
\end{corollary}

\begin{proof}
    We prove by induction on $\ell$. \cref{lem:lim-smallest-ev} implies the base case where $\ell = 0$. For the inductive step, suppose that $\ell \in \N^+$. Let $r$ be the vertex in $V(G_S, \ell-1) \setminus V(G_S,\ell-2)$. Define the rooted graph $H_r$ to be the graph $H = (G_S, \ell-1)$ equipped with a single root $r$. Note that $(H_r, m) = (G_S, \ell + m)$, and so $\lim_{m\to\infty}\la_1(H_r, m) = -\la$. Apply the base case to $H_r$, we know that $A_{(H_r,0)} + \la I - (\la/2-\sqrt{\la^2/4-1})E_{v_\ell,v_\ell}$ is singular.
\end{proof}

Now we obtain a simple matrix criterion to decide whether $\la_1(F_R, \ell, G_S) > \lim_{\ell \to \infty}\la_1(G_S, \ell)$, from which \cref{lem:reduction-gen} follows immediately.

\begin{lemma} \label{thm:matrix}
    Suppose that $G_S$ is a rooted graph. If $\la = -\lim_{\ell \to \infty}\la_1(G_S, \ell)$, and $\la_1(G)> -\la$, then the following holds. For every rooted graph $F_R$ and $\ell \in \N$, the smallest eigenvalue of the path augmentation $(F_R, \ell, G_S)$ is more than $-\la$ if and only if the matrix
    \[
        A_{(F_R, 0)} + \la I - \left(\la/2 + \sqrt{\la^2/4-1}\right) E_{v_0,v_0}
    \]
    is positive definite, where $v_0$ is the vertex in $V(F_R, 0) \setminus V(F)$.
\end{lemma}

\begin{proof}
    We claim that $\la_1(G_S, \ell-1) > -\la$. Assume for the sake of contradiction that $\la_1(G_S, \ell-1) = -\la$. Let $\bx \colon V(G_S,\ell-1) \to \R$ be an eigenvector of $(G_S, \ell-1)$ associated with $-\la$. Let $v_0 \dots v_\ell$ be the path of length $\ell$ that is added to $G$ to obtain $(G_S, \ell)$. We extend $\bx$ to $\tbx \colon V(G_S, \ell) \to \R$ by setting $\tilde{x}_{v_\ell} = 0$. Since $\tbx^\T A_{(G_S,\ell)}\tbx = -\la \tbx^\T\tbx$, and $\la_1(G_S,\ell) = -\la$, the nonzero vector $\tbx$ is an eigenvector of $(G_S,\ell)$. Using $-\la x_{v_i} = \sum_{u \sim v_i} x_u$ for $i \in \sset{\ell, \dots, 1}$, one can then prove by induction that $x_{v_i} = 0$ for $i \in \sset{\ell, \dots, 0}$. Therefore the vector $\bx$ restricted to $V(G)$ is an eigenvector of $G$ associated with $-\la$, which contradicts the assumption that $\la_1(G) > -\la$.

    Coming back to the proof of the lemma, we partition the matrix $A_{(F_R, \ell, G_S)} + \la I$ into the following blocks:
    \[
        \begin{pmatrix}
            A_{(F_R,0)}+\la I & B \\
            B^\T & C
        \end{pmatrix}.
    \]
    Since $C = A_{(G_S, \ell-1)}+\la I$, and $\la_1(G_S, \ell-1) > -\la$ from the claim, the block $C$ is positive definite. Therefore, the above block matrix is positive definite if and only if the Schur complement $A_{(F_R,0)} + \la I - BC^{-1}B^\T$ of $C$ is positive definite. Since the only nonzero row of $B$ is its $v_0$-th row, say $B_{v_0}$, the matrix $BC^{-1}B^\T$ simplifies to $\left(B_{v_0}C^{-1}B_{v_0}^\T\right) E_{v_0,v_0}$.
    
    It suffices to verify that $B_{v_0} C^{-1}B_{v_0}^\T = \la/2 + \sqrt{\la^2 / 4 - 1}$. Notice that the block matrix
    \[
        \begin{pmatrix}
            \la/2 + \sqrt{\la^2 / 4 - 1} & B_{v_0} \\
            B_{v_0}^\T & C    
        \end{pmatrix}
    \] is precisely the matrix in \cref{cor:ga}, and so it is singular. Since the block $C$ is invertible, its Schur complement $\la/2 + \sqrt{\la^2 / 4 - 1} - B_{v_0}C^{-1}B_{v_0}^\T$ must be zero.
\end{proof}

\begin{proof}[Proof of \cref{lem:reduction-gen}]
    Simply notice that the matrix criterion in \cref{thm:matrix} is independent of $\ell$.
\end{proof}

\section{Characterization of rooted graphs} \label{sec:char}

In this section, we characterize those rooted graphs $F_R$ for which $\ape{F_R, 0} \in \G(\las)$. To this end, we first introduce the notion of a single-rooted graph and its line graph.

\begin{definition}[Single-rooted graph and its line graph]
    A \emph{single-rooted graph} $H_r$ is a rooted graph $H$ with a single root $r$. The \emph{line graph} of a single-rooted graph $H_r$, denoted by $L(H_r)$, is the rooted graph $F_R$, where $F$ is the line graph of $H$, and $R$ is the set of edges incident to $r$ in $H$.
\end{definition}

\begin{theorem} \label{thm:main2}
    There exists a finite family $\F$ of rooted graphs such that for every connected augmented path extension, it is an augmented path extension of a rooted graph in $\F$ if and only if its smallest eigenvalue is greater than $-\las$. Moreover, the family $\F$ has the following properties:
    \begin{enumerate}[label=(\alph*)]
        \item every rooted graph in $\F$ is the line graph of a connected bipartite single-rooted graph $H_r$, for which $r$ is not a leaf of $H$, \label{item:main2-a}
        \item for every rooted graph $F_R$ in $\F$, there exists $\ell_0 \in \sset{0,\dots,6}$ such that the smallest eigenvalue of $\ape{F_R, \ell}$ is in $(-\las,-2)$ if and only if $\ell \ge \ell_0$. \label{item:main2-c}
    \end{enumerate}
\end{theorem}

\begin{remark}
    The finiteness of the family $\F$ is not essential here, since it will be confirmed by the enumeration in \cref{sec:enum}. We include it here because the proof of \cref{thm:main2} already establishes finiteness with only minor computer assistance for small computations.
\end{remark}

To characterize such rooted graphs $F_R$, we need the following sufficient condition for line graphs, which is an immediate consequence of \cite[Theorem~4]{RW65}. The claw graph and the diamond graph are defined in \cref{fig:claw-and-diamond}.

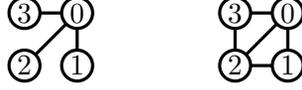
\begin{figure}
    \centering
    \begin{tikzpicture}[very thick, scale=0.7]
        \begin{scope}[shift={(4,0)}]
            \draw (-1,1) -- (0,1) -- (0,0);
            \draw (0,1) -- (-1,0);
            \node[circ] at (0,1) {0};
            \node[circ] at (0,0) {1};
            \node[circ] at (-1,0) {2};
            \node[circ] at (-1,1) {3};
        \end{scope}
        \begin{scope}[shift={(8,0)}]
            \draw (-1,1) -- (0,1) -- (0,0) -- (-1,0) -- cycle;
            \draw (0,1) -- (-1,0);
            \node[circ] at (0,1) {0};
            \node[circ] at (0,0) {1};
            \node[circ] at (-1,0) {2};
            \node[circ] at (-1,1) {3};
        \end{scope}
    \end{tikzpicture}
    \caption{The claw graph $C$ and the diamond graph $D$.} \label{fig:claw-and-diamond}
\end{figure}

\begin{theorem}[Theorem~4 of van Rooij and Wilf~\cite{RW65}] \label{thm:claw-diamond}
    Every graph that contains neither the claw graph nor the diamond graph as a subgraph is a line graph. \qed
\end{theorem}

The following computation, together with \cref{thm:reduction}, roughly speaking, enables us to forbid certain subgraphs in the rooted graph $F_R$.

\begin{proposition} \label{lem:forbidden-rooted-graph}
    For every rooted graph $F_R$ in \cref{fig:forbidden-rooted-graphs}, the smallest eigenvalue of $\ape{F_R, 0}$ is less than $-\las$.
\end{proposition}

\begin{figure}[b]
    \centering
    \begin{minipage}[t]{0.25\textwidth} \centering
    \begin{tikzpicture}[very thick, scale=0.7, baseline=(v.base)]
        \coordinate (v) at (0,0);
        \node[root] at (0,1) {0};
        \node[root] at (0,0) {1};
        \node at (0, -1) {$\overline{K}_2$};
        \node at (0, -1.5) {};
    \end{tikzpicture}
    \end{minipage}%
    \begin{minipage}[t]{0.25\textwidth} \centering
        \begin{tikzpicture}[very thick, scale=0.7, baseline=(v.base)]
            \draw (-1,1) -- (0,1) -- (0,0);
            \draw (0,1) -- (-1,0);
            \node[root] at (0,1) {0};
            \node[root] at (0,0) {1};
            \node[circ] at (-1,0) {2};
            \node[circ] at (-1,1) {3};
            \node at (-.5, -1) {$S_3$};
        \end{tikzpicture}
    \end{minipage}%
    \begin{minipage}[t]{0.25\textwidth} \centering
        \begin{tikzpicture}[very thick, scale=0.7, baseline=(v.base)]
            \coordinate (v) at (0,0);
            \draw (0,1) -- (0,0) -- (-1,0) -- cycle;
            \node[root] at (0,1) {0};
            \node[root] at (0,0) {1};
            \node[circ] at (-1,0) {2};
            \node at (-0.5, -1) {$K_3$};
        \end{tikzpicture}
    \end{minipage}%
    \begin{minipage}[t]{0.25\textwidth} \centering
        \begin{tikzpicture}[very thick, scale=0.7, baseline=(v.base)]
            \draw (0,0) -- (1,0) -- (2,0) -- (2,1) -- (1,1) -- cycle;
            \node[circ] at (0,0) {3};
            \node[circ] at (1,0) {2};
            \node[root] at (2,0) {1};
            \node[root] at (2,1) {0};
            \node[circ] at (1,1) {4};
            \node at (1, -1) {$C_5$};
        \end{tikzpicture}
    \end{minipage}
    \begin{minipage}[t]{0.25\textwidth} \centering
        \begin{tikzpicture}[very thick, scale=0.7, baseline=(v.base)]
            \draw (0,0) -- (1,0) -- (2,0) -- (3,0) -- (3,1) -- (2,1) -- (1,1) -- cycle;
            \node[circ] at (0,0) {4};
            \node[circ] at (1,0) {3};
            \node[circ] at (2,0) {2};
            \node[root] at (3,0) {1};
            \node[root] at (3,1) {0};
            \node[circ] at (2,1) {6};
            \node[circ] at (1,1) {5};
            \node at (1.5, -1) {$C_7$};
        \end{tikzpicture}
    \end{minipage}%
    \begin{minipage}[t]{0.25\textwidth} \centering
        \begin{tikzpicture}[very thick, scale=0.7, baseline=(v.base)]
            \draw (0,0) -- (1,0) -- (2,0) -- (3,0) -- (3,1) -- (2,1) -- (1,1) -- (0,1);
            \node[circ] at (0,0) {4};
            \node[circ] at (1,0) {3};
            \node[circ] at (2,0) {2};
            \node[root] at (3,0) {1};
            \node[root] at (3,1) {0};
            \node[circ] at (2,1) {7};
            \node[circ] at (1,1) {6};
            \node[circ] at (0,1) {5};
            \node at (1.5, -1) {$P_7$};
        \end{tikzpicture}
    \end{minipage}%
    \begin{minipage}[t]{0.25\textwidth} \centering
        \begin{tikzpicture}[very thick, scale=0.7, baseline=(v.base)]
            \coordinate (v) at (0,0);
            \draw (0,1) -- (0,0) -- (-1,0) -- (-1,1) -- (-2,1) -- (-2,0) -- (-3,0) -- (-3,1) -- (-4,1) -- (-4,0);
            \draw (-1,0) ;
            \node[root] at (0,1) {0};
            \node[root] at (0,0) {1};
            \node[circ] at (-1,0) {2};
            \node[circ] at (-1,1) {3};
            \node[circ] at (-2,1) {4};
            \node[circ] at (-2,0) {5};
            \node[circ] at (-3,0) {6};
            \node[circ] at (-3,1) {7};
            \node[circ] at (-4,1) {8};
            \node[circ] at (-4,0) {9};
            \node at (-2, -1) {$P_9$};
        \end{tikzpicture}
    \end{minipage}%
    \begin{minipage}[t]{0.25\textwidth} \centering
        \begin{tikzpicture}[very thick, scale=0.7, baseline=(v.base)]
            \coordinate (v) at (0,-1);
            \foreach \x in {0,1,...,6} {
                \coordinate (node\x) at (-\x*360/7+90/7:1);
            }
            \foreach \x in {0,1,...,5} {
                \foreach \y in {\x,...,6} {
                    \draw (node\x.center) -- (node\y.center);
                }
            }
            \foreach \x in {0,1,...,6} {
                \node[root] at (node\x) {\x};
            }
            \node at (0, -2) {$K_7$};
        \end{tikzpicture}
    \end{minipage}
    \caption{The rooted graphs in \cref{lem:forbidden-rooted-graph}.} \label{fig:forbidden-rooted-graphs}
\end{figure}

\begin{cproof}
    We shall prove that $\la_1\ape{F_R, 0} < -101/50 = -2.02$. The computation is straightforward. We label the rooted graphs in \cref{fig:forbidden-rooted-graphs} by \texttt{K2C, S3, K3, C5, C7, P7, P9, K7}. For each rooted graph $F_R$, we output the determinant of $A_{\apesub{F_R,0}} + (101/50)I$, which turns out to be negative. Our code is available as the ancillary file \texttt{forb\_rooted\_graphs.rb} in the arXiv version of the paper.
\end{cproof}

\begin{lemma} \label{lem:char}
    For every rooted graph $F_R$, if $(F_R, 0)$ is connected and $\la_1\ape{F_R, 0} > -\las$, then there exists a connected bipartite single-rooted graph $H_r$ such that
    \begin{enumerate}[label=(\alph*)]
        \item $F_R$ is the line graph of $H_r$; \label{item:a}
        \item if $r$ is not a leaf of $H$, then every vertex of $H$ is at distance at most $8$ from $r$; \label{item:b}
        \item the maximum degree of $H$ is at most $7$. \label{item:c}
    \end{enumerate}
\end{lemma}

\begin{proof}
    Suppose that $F_R$ is a rooted graph such that $(F_R, 0)$ is connected and $\la_1\ape{F_R, 0} > -\las$. Let $\overline{K}_2$, $S_3$, $K_3$, $C_5$, $C_7$, $P_7$, $P_9$ and $K_7$ be the eight rooted graphs defined in \cref{fig:forbidden-rooted-graphs}.

    We first prove that the claw graph $C$ is not a subgraph of $(F_R, 1)$. Assume for the sake of contradiction that $C$ is a subgraph of $(F_R, 1)$. Then $C$ is also a subgraph of $(F_R, 2)$. Let $v_2$ be the vertex in $V(F_R, 2) \setminus V(F_R, 1)$, and let $v_2 u_1\dots u_\ell$ denote a shortest path (possibly of length $0$) from $v_2$ to a vertex $u_\ell$ at distance $1$ from $C$ in $(F_R, 2)$. Label the vertices of $C$ as in \cref{fig:claw-and-diamond}, and let $S \subseteq \sset{0,1,2,3}$ be the nonempty subset of vertices of $C$ that are adjacent to $u_\ell$ in $(F_R, 2)$. Notice that the augmented path extension $\ape{F_R, 2}$ contains
    \begin{align*}
        \ape{\overline{K}_2,\ell+1} & \text{ when } S = \sset{0};\\
        \ape{\overline{K}_2,\ell+2} & \text{ when } \abs{S} = 1 \text{ and } 0 \not\in S;\\
        \ape{S_3,\ell} & \text{ when } \abs{S \cap \sset{1,2,3}}=1 \text{ and } 0 \in S; \\
        \ape{\overline{K}_2,\ell} & \text{ when } \abs{S \cap \sset{1,2,3}} \ge 2,
    \end{align*}
    as a subgraph, which yields a contradiction in view of \cref{thm:reduction,lem:forbidden-rooted-graph}.

    We next prove that the diamond graph $D$ is not a subgraph of $(F_R, 1)$. Assume for the sake of contradiction that $D$ is a subgraph of $(F_R, 1)$. Let $v_1$ be the vertex in $V(F_R, 1) \setminus V(F_R, 0)$, and let $v_1u_1\dots u_\ell$ be a shortest path (possibly of length $0$) from $v_1$ to a vertex $u_\ell$ at distance $1$ from $D$. Label the vertices of $D$ as in \cref{fig:claw-and-diamond}, and let $S\subseteq \sset{0,1,2,3}$ be the nonempty subset of vertices of $D$ that are adjacent to $u_\ell$ in $(F_R, 1)$. Notice that the augmented path extension $\ape{F_R, 1}$ contains
    \begin{align*}
        \ape{\overline{K}_2,\ell+1} & \text{ when } S = \sset{0} \text{ or }S = \sset{2}; \\
        \ape{K_3, \ell+1} & \text{ when } S = \sset{1} \text{ or }S = \sset{3}; \\
        \ape{\overline{K}_2,\ell} & \text{ when } S \supseteq \sset{1,3}; \\
        \ape{K_3, \ell} & \text{ when } \abs{S} \ge 2 \text{ and } S \not\supseteq \sset{1,3},
    \end{align*}
    as a subgraph, which yields a contradiction in view of \cref{thm:reduction,lem:forbidden-rooted-graph}.

    At this point, \cref{thm:claw-diamond} implies that there exists a graph $H'$ such that $(F_R, 1)$ is the line graph of $H'$. Let $v_0$ be the vertex in $V(F_R, 0) \setminus V(F)$, and let $v_1$ be the vertex in $V(F_R, 1) \setminus V(F_R, 0)$. We identify the two vertices $v_0$ and $v_1$ of $(F_R, 1)$ with two edges $e_0$ and $e_1$ of $H'$. Since $v_0v_1$ is an edge of $(F_R, 1)$, the edges $e_0$ and $e_1$ share a common vertex in $H'$. Let $r, u_0, u_1$ be the vertices of $H'$ such that $e_0 = ru_0$ and $e_1 = u_0u_1$. Since $v_1$ is a leaf of $(F_R, 1)$ and $v_0v_1$ is a pendant edge of $(F_R, 1)$, we deduce that $u_1$ is a leaf of $H'$, and $u_0$ is a leaf of $H' - u_1$. Let $H = H' - \sset{u_0, u_1}$. Clearly $F_R$ is the line graph of $H_r$. Since $(F_R, 1)$ is connected, so are $H'$ and $H$.
    
    To finish the proof of \ref{item:a}, we need to further show that $H$ is bipartite. Assume for the sake of contradiction that $H$ contains an odd cycle $C_k$ of length $k$ as a subgraph that is not necessarily induced. Take a shortest path $P$ of length $\ell$ between $r$ and $C_k$. Notice that the edges in $P$ and $C_k$ induce the following graph as a subgraph of $\ape{F_R, 0}$:
    \begin{align*}
        \ape{K_3, \ell} & \text{ when } k = 3; \\
        \ape{C_5, \ell} & \text{ when } k = 5; \\
        \ape{C_7, \ell} & \text{ when } k = 7; \\
        \ape{P_7, \ell} & \text{ when } k \ge 9,
    \end{align*}
    which yields a contradiction in view of \cref{thm:reduction,lem:forbidden-rooted-graph}.

    We are left to prove \ref{item:b} and \ref{item:c}. Assume for the sake of contradiction that $r$ is not a leaf of $H$, and there exists a vertex $u_9$ at distance $9$ from $r$ in $H$. Take a shortest path $P := ru_1\dots u_9$ between $r$ and $u_9$. Since $r$ is not a leaf of $H$, we can choose a neighbor of $r$, say $u_0$, in $H$ that is not on $P$. Notice that these edges $ru_0, ru_1, u_1u_2, \dots, u_8u_9$ in $H$ induce $\ape{P_9, 0}$ as a subgraph of $\ape{F_R, 0}$, which yields a contradiction in view of \cref{lem:forbidden-rooted-graph}. Lastly, assume for the sake of contradiction that there exists a vertex $u$ with degree at least $8$ in $H$. Take a shortest path $P$ of length $\ell$ between $r$ and $u$ in $H$. We can choose neighbors of $u$, say $u_1, \dots, u_7$, in $H$ that are not on $P$. Notice that the edges of $P$ together with the edges $uu_1, \dots, uu_7$ induce $\ape{K_7, \ell}$ as a subgraph of $\ape{F_R, 0}$, which yields a contradiction in view of \cref{thm:reduction,lem:forbidden-rooted-graph}.
\end{proof}

We need one more ingredient on augmented path extensions that are not in $\G(2)$.

\begin{lemma} \label{lem:ape-not-2}
    For every rooted graph $F_R$ and every $\ell \in \N$, if the smallest eigenvalue of $\ape{F_R, \ell}$ is less than $-2$, then the same holds for $\ape{F_R, \ell+1}$.
\end{lemma}

\begin{proof}
    Let $v_0\dots v_\ell$ be the path of length $\ell$ added to the disjoint union of $F$ and the rooted graph \etwo\ to obtain $\ape{F_R, \ell}$, and let $v_{-1}v_0\dots v_\ell$ be the corresponding path of length $\ell+1$ to obtain $\ape{F_R, \ell+1}$. Suppose that $\la_1\ape{F_R, \ell} < -2$. We can pick a nonzero vector $\bx \colon V\ape{F_R,\ell} \to \R$ such that $\bx^\T A_{\apesub{F_R,\ell}}\bx < -2 \bx^\T\bx$. Define the vector $\tbx \colon V\ape{F_R, \ell+1} \to \R$ by
    \[
        \tilde{x}_v = \begin{cases}
            -x_v & \text{if }v \in V(F); \\
            -x_{v_0} & \text{if }v = v_{-1}; \\
            x_v & \text{otherwise}. \\
        \end{cases}
    \]
    Clearly, $\tbx^\T A_{\apesub{F_R, \ell+1}}\tbx = \bx^\T A_{\apesub{F_R, \ell}} \bx - 2x_{v_0}^2$ and $\tbx^\T \tbx = \bx^\T \bx + x_{v_0}^2$. In particular, $\tbx$ is a nonzero vector that satisfies $\tbx^\T A_{\apesub{F_R, \ell+1}}\tbx < -2\tbx^\T \tbx$, which implies that $\la_1\ape{F_R, \ell+1} < -2$ according to the Rayleigh principle.
\end{proof}

We are in the position to characterize the rooted graphs $F_R$ for which $\ape{F_R, 0} \in \G(\las)$.

\begin{proof}[Proof of \cref{thm:main2}] \label{proof:main}
    Let the family $\F$ consist of rooted graphs $F_R$ such that $(F_R, 0)$ is connected, $\la_1\ape{F_R, \ell} > -\las$ for every $\ell \in \N$, and $F_R = K_0$ or $\abs{R} \ge 2$. Consider a connected augmented path extension of a rooted graph $F_R$ in $\G(\las)$. Without loss of generality, we may assume that either $F_R = K_0$ or $\abs{R} \ge 2$. Clearly, $(F_R, 0)$ is connected. By \cref{thm:reduction}, $\la_1\ape{F_R, \ell} > -\las$ for every $\ell \in \N$, and so $F_R \in \F$.

    For \ref{item:main2-a}, apply \cref{lem:char} to each rooted graph $F_R \in \F$ to obtain a connected bipartite single-rooted graph $H_r$ such that $F_R$ is the line graph of $H_r$. Since $R = \varnothing$ or $\abs{R} \ge 2$, the root $r$ of $H_r$ cannot be a leaf of $H$. Furthermore, \cref{lem:char}\ref{item:b} and \cref{lem:char}\ref{item:c} imply that the order of $H_r$ is bounded by $7^8$, implying that $\F$ is finite.

    For \ref{item:main2-c}, take a rooted graph $F_R$ in $\F$. Note that $\ape{F_R,6}$ contains $E_{10}$ as a subgraph, whose smallest eigenvalue is less than $-2$. Let $\ell_0 \in \N$ be the minimum $\ell \in \N$ such that $\la_1\ape{F_R,\ell} < -2$. In particular, $\ell_0 \le 6$. By \cref{lem:ape-not-2}, $\la_1\ape{F_R, \ell} < -2$ if and only if $\ell \ge \ell_0$.
\end{proof}

\section{Enumeration of the rooted graphs} \label{sec:enum}

In this section, we enumerate, by means of a computer search, all connected bipartite single-rooted graphs $H_r$ such that $r$ is not a leaf of $H$, and $\la_1\ape{L(H_r), 0} > -\las$, and we describe them concisely through their maximal members. We also record, for each such $H_r$, the minimum value of $\ell_0 \in \sset{0,\dots,6}$ for which $\la_1\ape{L(H_r),\ell_0} < -2$. By \cref{thm:main2}, each pair $(H_r, \ell_0)$ thus determines a family of augmented path extensions $\ape{F_R, \ell}$ in $\G(\las) \setminus \G(2)$, where $F_R = L(H_r)$ and $\ell \ge \ell_0$.

\begin{definition}
    Given two single-rooted graphs $H_r$ and $H'_s$, we say that $H_r$ is a \emph{general subgraph} of $H'_s$ if there exists an injective graph homomorphism from $H$ to $H'$ that maps $r$ to $s$.
\end{definition}

\begin{theorem} \label{thm:enumeration}
    Let the family $\HH$ consist of the connected bipartite single-rooted graphs $H_r$ such that $r$ is not a leaf of $H$, and the smallest eigenvalue of $\ape{L(H_r), 0}$ is greater than $-\las$. Let $\HH^*$ be the subfamily of $\HH$ that consists members that are maximal under general subgraphs.
    \begin{enumerate}[label=(\alph*)]
        \item For every general subgraph $H_r$ of a member in $\HH^*$, if $H$ is connected and $r$ is not a leaf of $H$, then $H_r$ is in $\HH$. \label{item:enum-a}
        \item There are a total of 794 members of $\HH$ with the following statistics.\footnote{The \emph{size} of a graph is the number of its edges.}
        \begin{center}
\begin{tabular}{ccccccccccccccc}
                size & 0 & 2 & 3 & 4 & 5 & 6 & 7 & 8 & 9 & 10 & 11 & 12 & 13 & 14 \\
                \hline
                \# & 1 & 1 & 2 & 6 & 14 & 42 & 107 & 190 & 194 & 136 & 68 & 27 & 4 & 2
            \end{tabular}
        \end{center}
        \item There are a total of 48 members of $\HH^*$, which are listed in \cref{fig:maximal-root-graphs}.
    \end{enumerate}
\end{theorem}

\begin{figure}
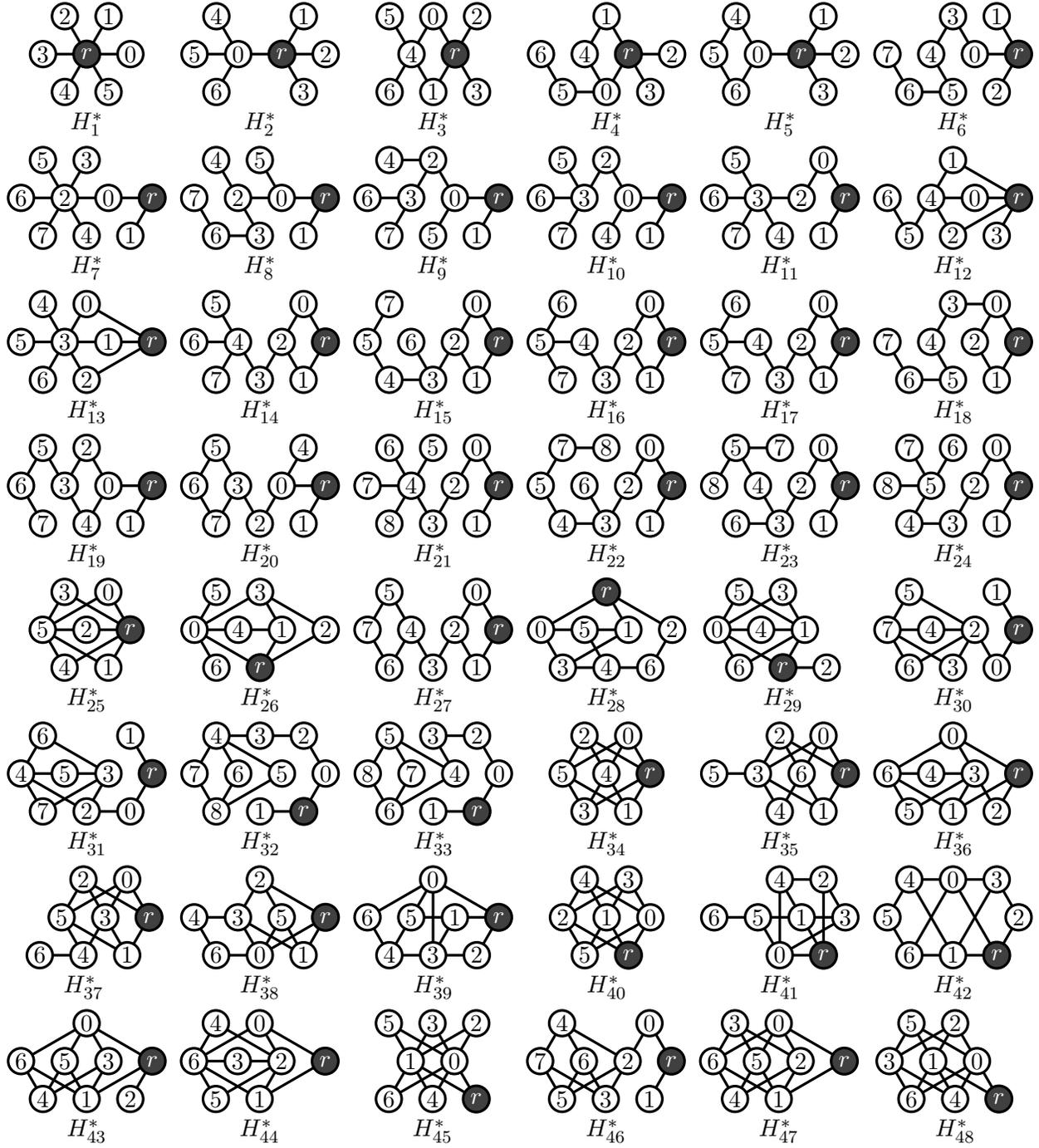

    \begin{minipage}[t]{0.16\textwidth} \centering
    %
    \end{minipage}
    \caption{Single-rooted graphs of $\HH$ that are maximal under general subgraphs.} \label{fig:maximal-root-graphs}
\end{figure}

\begin{cproof}
    For \ref{item:enum-a}, consider a general subgraph $H_r$ of $H_s' \in \HH$. Since $\ape{L(H_r),0}$ is a subgraph of $\ape{L(H_s'),0}$, we know that $\la_1\ape{L(H_r),0} > -\las$.

    To enumerate the members in $\HH$, notice that for every member $H_r \in \HH$, except the trivial single-rooted graph $K_1$, there exists a sequence $H_r^{(2)}, \dots, H_r^{(n)} = H_r$ of members of $\HH$ such that $H_r^{(2)} = \stwo$, and for every $i \in \sset{2,\dots, n-1}$, $H^{(i+1)}$ is obtained from $H^{(i)}$ by adding an edge that is incident to at least one vertex of $H^{(i)}$. This allows us to search for more non-trivial members of $\HH$ by adding a new edge to the existing ones.
    
    We store members of $\HH$ in the hash \texttt{dict} based on their size, and store those that are maximal under general subgraphs in the array \texttt{maximal}. At the start, \texttt{dict[0]} consists of the trivial single-rooted graph $K_1$, \texttt{dict[2]} consists of \stwo, and the counter \texttt{m} is incremented from \texttt{2}.
    
    Whenever \texttt{dict[m]} is nonempty, we iterate through members of \texttt{dict[m]}. For each $H_r$ in \texttt{dict[m]}, we carry out the following steps.
    \begin{enumerate}[label=(\roman*)]
        \item We add a new edge to $H_r$ in every possible way to obtain new connected bipartite single-rooted graphs $H_r'$.
        \item We admit those $H_r'$ with $\la_1{\ape{L(H_r'), 0}} > -\las$ to \texttt{dict[m+1]} (cf.\ \cref{subsec:1,subsec:3}).
        \item We append $H_r$ to \texttt{maximal} when no $H_r'$ was admitted to \texttt{dict[m+1]}.
    \end{enumerate}
    
    For each $H_r$ in \texttt{dict}, we output it as a string of the form \texttt{u[1]u[2]...u[2e-1]u[2e]} which lists the edges \texttt{u[1]u[2],...,u[2e-1]u[2e]} of $H_r$, and we designate \texttt{r} to represent the root of $H_r$ in the string. We also output the minimum $\ell_0 \in \sset{0,\dots, 6}$ such that $\la_1{\ape{L(H_r), \ell_0}} < -2$ (cf.\ \cref{subsec:2}). Finally, we output the maximal single-rooted graphs stored in \texttt{maximal}.

    Our code and its output are available as the ancillary files \texttt{enum\_rooted\_graphs.rb} and\linebreak \texttt{enum\_rooted\_graphs.txt} in the arXiv version of the paper.
\end{cproof}

For the rest of this section, we share further details of our implementation.

\subsection{Positive definiteness of \texorpdfstring{$A_{G'} + \las I$}{A\_G + λ*I}} \label{subsec:1}

Before we admit a single-rooted graph $H_r'$ to \texttt{dict[m+1]}, we need to check whether $A_{G'} + \las I$ is positive definite, where $G' = \ape{L(H_r'), 0}$. Since $H_r'$ is obtained from $H_r$ in \texttt{dict[m]} by adding a new edge, the graph $G := \ape{L(H_r), 0}$ can be obtained from $G'$ by removing a vertex. Since $A_G + \las I$ is already positive definite, by Sylvester's criterion, it suffices to check whether the determinant of $A_{G'} + \las I$ is positive. To avoid the irrational number $\las$, we use two rational approximations $\las_-$ and $\las_+$:
\[
    2.0198008850 \approx 18259/9040 =: \las_- < \las < \las_+ := 91499/45301 \approx 2.0198008874,
\]
Exactly one of the following two cases occurs.

\paragraph{Case 1:} $\det(A_{G'} + \las_- I) > 0$. Since $A_G + \las_- I$ is positive definite, so is $A_{G'} + \las_- I$ according to Sylvester's criterion. In this case, we can assert that $A_{G'} + \las I$ is positive definite.

\paragraph{Case 2:} $\det(A_{G'} + \las_+ I) < 0$. Since the matrix $A_{G'} + \las_+ I$ is not positive definite, we can assert that $A_{G'} + \las I$ is not positive definite either.

\bigskip

Otherwise, we raise an exception, although this never occurs for any graph encountered throughout the computer search.

\subsection{Hash function of single-rooted graphs} \label{subsec:3}

When we admit a single-rooted graph $H_s'$ to \texttt{dict[m+1]}, we need to check whether it is isomorphic to an existing member of \texttt{dict[m+1]}. To efficiently detect isomorphic duplicates, we maintain a hash table \texttt{@hash} of existing members of \texttt{dict[m+1]} using the following hash function.

For a bipartite single-rooted graph $H_r$, its hash value is a triple \texttt{[dr, dA, dB]}, where \texttt{dr} is the degree of $r$ in $H$, \texttt{dA} is the degree sequence of the vertices in the part that contains $r$, and \texttt{dB} is the degree sequence of the vertices in the part that does not contain $r$. Clearly, when two bipartite single-rooted graphs are isomorphic, their hash values are equal.

This allows us to test whether $H_s'$ is isomorphic to any existing member of \texttt{dict[m+1]} by examining only the members of \texttt{@hash[hv]}, where \texttt{hv} is the hash value of $H_s'$.

\subsection{Positive semidefiniteness of \texorpdfstring{$A_G + 2I$}{A\_G + 2I}} \label{subsec:2}

When computing the minimum $\ell_0$, Sylvester's criterion implies that checking whether $A_G + 2I$ is positive semidefinite amounts to verifying that all principal minors of $A_G + 2I$ are nonnegative. To make this subroutine more efficient, we use the following fact.

\begin{theorem}[Theorem~2.5 of Bussemaker and Neumaier~\cite{BN92}] \label{cor:two}
    There is no graph whose smallest eigenvalue is in $(\la_1(E_{10}),-2)$, where $\la_1(E_{10}) \approx -2.006594$. \qed
\end{theorem}

\begin{corollary}
    For every graph $G$, the matrix $A_G + 2I$ is positive semidefinite if and only if the matrix $A_G + (305/152) I$ is positive definite.
\end{corollary}

\begin{proof}
    Since $\la_1(E_{10}) < -305/152 < -2$, the smallest eigenvalue of $G$ is at least $-2$ precisely when it exceeds $-305/152$.
\end{proof}

In our implementation, we assert that $A_G + 2I$ is positive semidefinite if and only if all the \emph{leading} principal minors of $A_G + (305/152)I$ are positive.

\section{Enumeration of the maverick graphs} \label{sec:maverick}

In this section, we enumerate the connected graphs in $\G(\las)\setminus\G(2)$ that are not augmented path extensions.

\begin{theorem} \label{thm:main3}
    There are a total of 4752 maverick graphs with the following statistics.
    \begin{center}
        \begin{tabular}{cccccccccccc}
            order & 9 & 10 & 11 & 12 & 13 & 14 & 15 & 16 & 17 & 18 & 19 \\
            \hline
            \# & 13 & 629 & 1304 & 1237 & 775 & 408 & 221 & 107 & 42 & 13 & 3
        \end{tabular}
    \end{center}
\end{theorem}

We need the following technical result on the generation of maverick graphs.

\begin{definition}[Witness]
    Given a graph $G$, a quadruple $(u_0,u_1,u_2,u_c)$ of distinct vertices is a \emph{witness for an augmented path extension} if $u_0u_1$, $u_1u_2$ and $u_0u_c$ are the only edges of $G$ that are not in $G - \sset{u_1,u_2,u_c}$.
\end{definition}

\begin{remark}
    As the name suggests, a graph $G$ has a witness for an augmented path extension if and only if $G$ is an augmented path extension of a rooted graph.
\end{remark}

\begin{lemma} \label{lem:gen-mav}
    For every maverick graph $M$ on $n$ vertices, there exists a sequence $K_2 = G_2, \dots, G_n = M$ of connected graphs satisfying the following three properties.
    \begin{enumerate}[label=(\roman*)]
        \item For every $i \in \sset{3,\dots,n}$, $\la_1(G_i) > -\las$, and there exists $v_i \in V(G_i)$ such that $G_{i-1} = G_i - v_i$.
        \item If $n \ge 10$, then $\la_1(G_{10}) < -2$. \label{item:res-1}
        \item If $n \ge 11$ and $G_{10}$ has a unique witness for an augmented path extension, then $G_{11}$ is not an augmented path extension. \label{item:res-2}
    \end{enumerate}
\end{lemma}

The proof requires the following fact about graphs with smallest eigenvalue less than $-2$ that are minimal under subgraphs.

\begin{theorem}[Kumar, Rao, and Singhi~\cite{KRS82}] \label{thm:min-forb-10}
    Every minimal forbidden subgraph for the family $\G(2)$ of graphs with smallest eigenvalue at least $-2$ has at most 10 vertices. \qed
\end{theorem}

\begin{proof}[Proof of \cref{lem:gen-mav}]
    Suppose $M$ is a maverick graph on $n$ vertices. Since $M$ is connected, $M$ is not a null graph, and $\la_1(M) > -\las$, the case where $n \le 9$ is trivial. Hereafter, we may assume that $n \ge 10$. \cref{thm:min-forb-10} provides a connected subgraph $G_{10}$ of $M$ on $10$ vertices such that $\la_1(G_{10}) < -2$. We can then easily build the other connected graphs in the sequence using $G_{10}$.

    In the case where $n \ge 11$ and $G_{10}$ has a unique witness, say $(u_0,u_1,u_2,u_c)$, of an augmented path extension, since $M$ is not an augmented path extension, there exists a vertex $v_{11} \in V(M) \setminus V(G_{10})$ that is adjacent to at least one of $u_1$, $u_2$ and $u_c$ in $M$. We can specifically choose $G_{11} = M[V(G_{10}) \cup \sset{v_{11}}]$. Assume for the sake of contradiction that $G_{11}$ has a witness $(u_0',u_1',u_2',u_c')$. It must be the case that $u_0'=u_0$, $u_1'=u_c$, $u_2'=v_{11}$, $u_c' \not\in \sset{u_0,u_1,u_2,u_c}$, and $u_0u_1$, $u_1u_2$, $u_0u_c$, $u_cv_{11}$ and $u_0u_c'$ are the only edges of $G_{11}$ that are not in $G_{11} - \sset{u_1, u_2, u_c, u_c'}$. Thus $G_{10}$ has another witness $(u_0,u_1,u_2,u_c')$, which contradicts the uniqueness of the witness $(u_0,u_1,u_2,u_c)$ for $G_{10}$.
\end{proof}

The enumeration of all maverick graphs is achieved by a computer search.

\begin{cproof}[Proof of \cref{thm:main3}]
    To enumerate the maverick graphs, we store the graphs that can possibly occur in a sequence described by \cref{lem:gen-mav} in the hash table \texttt{dict} based on their order. At the start, \texttt{dict[2]} consists of the single graph $K_2$, and the counter \texttt{n} is incremented from \texttt{2}.
    
    Whenever \texttt{dict[n]} is nonempty, we iterate through members of \texttt{dict[n]}. For each $G$ in \texttt{dict[n]}, we carry out the following five steps.

    \begin{enumerate}[label=(\roman*)]
        \item We connect a new vertex to a nonempty vertex subset $S$ of $G$ in every possible way to obtain new graphs $G' = (G_S, 0)$. See \cref{subsec:6} for a more efficient implementation.
        \item We store those $G'$ with $\la_1(G') > -\las$ in a temporary array \texttt{candidates}.
        \item In view of \cref{lem:gen-mav}\ref{item:res-1}, when \texttt{n = 9}, we remove $G'$ from \texttt{candidates} when $\la_1(G') < -2$.
        \item In view of \cref{lem:gen-mav}\ref{item:res-2}, when \texttt{n = 10}, we remove $G'$ from \texttt{candidates} when $G$ has a unique witness for an augmented path extension, and $G'$ is an augmented path extension (cf.\ \cref{subsec:5}).
        \item We merge \texttt{candidates} into \texttt{dict[n+1]} (cf.\ \cref{subsec:4}).
    \end{enumerate}

    We append the maverick graphs among the members in \texttt{dict[n]} to the array \texttt{mavericks}. To select the maverick graphs, we reject members of \texttt{dict[n]} with smallest eigenvalue at least $-2$ for \texttt{n} from \texttt{2} to \texttt{9}, and we always reject augmented path extensions in \texttt{dict[n]}.

    Our program terminates at \texttt{n = 20} because \texttt{dict[20]} is empty. Our code and its output are available as the ancillary files \texttt{enum\_maverick\_graphs.rb} and \texttt{enum\_maverick\_graphs.txt} in the arXiv version of the paper.
\end{cproof}

\begin{remark}
    On a MacBook Pro equipped with an Apple M1 Pro chip and 16 GB of memory, the program initially completed its task in under 25 minutes. With the assistance of ChatGPT-4, we rewrote the code in Julia, a high-performance dynamic language, reducing runtime to under 8 minutes on the same machine. The code is available as \texttt{enum\_maverick\_graphs.jl} in the arXiv version of the paper.
\end{remark}

We reuse the techniques discussed in \cref{subsec:1,subsec:2} to check positive definiteness and positive semidefiniteness of matrices. We share additional details of our implementation below.

\subsection{Adding a new vertex} \label{subsec:6}

To accelerate the generation of graphs, we leverage on the computation done for $G$ in \texttt{dict[n]} based on the following observation. Suppose that $G$ is a graph such that $\la_1(G) > -\las$. Let $\mathcal{S} = \dset{S \subseteq V(G)}{\la_1(G_S, 0) > -\las}$. For every $S \in \mathcal{S}$, and every $U \subseteq V(G')$, where $G' = (G_S, 0)$, note that $\la_1(G'_U, 0) > -\las$ implies that $\la_1(G_{U \cap V(G)}, 0) > -\las$, which is equivalent to $U \cap V(G) \in \mathcal{S}$. In other words, when we connect a new vertex to a nonempty subset $U$ of $G' = (G_S, 0)$ with $S \in \mathcal{S}$, we only need to iterate through $U$ with $U \cap V(G) \in \mathcal{S}$.

To keep track of the set $\mathcal{S}$ defined for $G$, we add an attribute \texttt{@possible\_subsets}, a list of distinct nonempty vertex subsets of $G$, to each graph $G$. For the graph $K_2$ in \texttt{dict[2]}, its \texttt{@possible\_subsets} consists of all the nonempty vertex subsets of $K_2$. For each $G$ in \texttt{dict[n]}, instead of connecting a new vertex, say $v'$, to a nonempty vertex subset $S$ of $G$ in every possible way, we connect $v'$ to $S$ for every $S$ in \texttt{@possible\_subsets} of $G$. To obtain \texttt{@possible\_subsets} of the graphs obtained from $G$, we initially set \texttt{new\_possible\_subsets} to be the list with a single vertex subset $\sset{v'}$. For each $S$ in \texttt{@possible\_subsets} of $G$, check whether $\la_1(G_S, 0) > -\las$, and if so, we append both $S$ and $S \cup \sset{v'}$ to \texttt{new\_possible\_subsets}, and we append $(G_S, 0)$ to \texttt{candidates}. Finally, we set \texttt{@possible\_subsets} of each graph $G'$ in \texttt{candidates} to \texttt{new\_possible\_subsets}.

\subsection{Finding witnesses} \label{subsec:5}

Given a graph $G$, to find its witnesses for an augmented path extension, we iterate through all edges $u_1u_2$ of $G$ with $d(u_1) = 2$ and $d(u_2) = 1$, check whether the vertex $u_0$, that is, the other neighbor of $u_1$, is adjacent to a leaf $u_c$, and we output all quadruples $(u_0, u_1, u_2, u_c)$.

\subsection{Hash function of graphs and generalized degrees} \label{subsec:4}

When we add a graph $G$ in \texttt{candidates} to \texttt{dict[n+1]}, we need to check whether it is isomorphic to an existing member of \texttt{dict[n+1]}. To that end, we maintain a hash table \texttt{@hash} of existing members of \texttt{dict[n+1]} using the following hash function.

For a vertex $v$ of a graph $G$, the \emph{generalized degree} of $v$ in $G$ is the pair \texttt{[dv, dw]}, where \texttt{dv} is the degree of $v$ in $G$, and \texttt{dw} is the number of edges in the subgraph of $G$ induced by the neighbors of $v$. The hash value of a graph $G$ is then the sorted sequence of generalized degrees. Clearly, when two graphs are isomorphic, their hash values are equal, and moreover, the generalized degree of a vertex is preserved under isomorphism.

This allows us to test isomorphism between $G$ and any existing member of \texttt{dict[n+1]} by examining only the members of \texttt{@hash[hv]}, where \texttt{hv} is the hash value of $G$. In addition, when we attempt to build an isomorphism between $G$ and a member $G'$ of \texttt{@hash[hv]}, we only map a vertex of $G$ to a vertex of $G'$ that has the same generalized degree.

\section{Twisted maverick graphs} \label{sec:twisted}

A visual examination of the maverick graphs reveals that a notable portion of them appear structurally similar.

\begin{definition}[Twisted path extension and twisted maverick graph]
    Given a rooted graph $F_R$ and $\ell \in \N$, the \emph{twisted path extension} of the rooted graph $F_R$ is the path augmentation $\tpe{F_R, \ell}$. See \cref{fig:twisted} for a schematic drawing. Given a graph $G$, a quadruple $(u_0,u_1,u_2,u_c)$ of distinct vertices is a \emph{witness for a twisted path extension} if $u_0u_1$, $u_0u_2$, $u_1u_2$ and $u_0u_c$ are the only edges of $G$ that are not in $G - \sset{u_1,u_2,u_c}$. A maverick graph is \emph{twisted} if it is a twisted path extension of a rooted graph.
\end{definition}

\begin{figure}
    \centering
    \begin{tikzpicture}[very thick, scale=0.5, baseline=(v.base)]
        \coordinate (v) at (0,0);
        \draw[rounded corners=14pt] (-7.8, -1) rectangle (-4, 1) {};
        \fill[darkgray] (-3,0) -- (-4.68,0.733) -- (-4.68,-0.733) -- cycle;
        \fill[darkgray] (-5, 0) circle (0.8);
        \draw[text=white] (-5,0) node{$R$};
        \draw (-7,0) node{$F$};
        \draw (3,1) node[root-vertex]{} -- (3,0);
        \draw[dashed] (-0.75,0) -- (1,0);
        \draw (-3,0) node[vertex]{} -- (-2,0) node[vertex]{} -- (-1,0) node[vertex]{};
        \draw (1,0) node[vertex]{} -- (2,0) node[vertex]{} -- (3,0) node[vertex]{} -- (4,0) node[root-vertex]{} -- (4,1) node[root-vertex]{} -- (3,0);
        \draw [braket] (4,0.2) -- (-3,0.2) node [black,midway,yshift=-15pt] {\footnotesize $\ell + 2$};
    \end{tikzpicture}
    \caption[The twisted path extension.]{The twisted path extension $\tpe{F_R, \ell}$.} \label{fig:twisted}
\end{figure}

A direct computer screening of the maverick graphs produced in \cref{sec:maverick} reveals that roughly a quarter of them are twisted.

\begin{theorem} \label{thm:main4}
    There are a total of 1161 twisted maverick graphs with the following statistics.
    \begin{center}
        \begin{tabular}{ccccccccccc}
            order  & 10 & 11 & 12 & 13 & 14 & 15 & 16 & 17 & 18 & 19 \\
            \hline
            \# & 48 & 133 & 220 & 236 & 210 & 162 & 96 & 40 & 13 & 3
        \end{tabular}
    \end{center}
\end{theorem}

\begin{cproof}
    For each maverick graph, we find its witnesses for a twisted path extension as follows: we iterate through all edges $u_1u_2$ of $G$ with $d(u_1) = d(u_2) = 2$, check whether $u_1$ and $u_2$ share a common neighbor, if so, check whether the common neighbor, say $u_0$, is adjacent to a leaf $u_c$, and output all quadruples $(u_0, u_1, u_2, u_c)$. Once a witness is found for a maverick graph, we assert that it is twisted. As it turns out, every maverick graph that is twisted has a unique witness. Our code and its output are available as the ancillary files \texttt{enum\_twisted\_mavericks.rb} and \texttt{enum\_twisted\_mavericks.txt} in the arXiv version of the paper.
\end{cproof}

For the proof of \cref{cor:main}, we need the following connection between twisted path extensions and augmented path extensions.

\begin{proposition} \label{lem:tpe-vs-ape}
    For every rooted graph $F_R$ and $\ell \in \N$, if $\la_1\tpe{F_R, \ell} \le -2$, then $\la_1\tpe{F_R, \ell} \le \la_1\ape{F_R, \ell}$.
\end{proposition}

\begin{proof}
    Let $\la > 1$, to be chosen later, and let $v_0$ be the vertex in $V(F_R,\ell)\setminus V(F_R,\ell-1)$. In the case where $\ell = 0$, instead let $v_0$ be the vertex in $V(F_R,\ell)\setminus V(F_R)$. Denote by $E_{v_0,v_0}$ the matrix unit where the $(v_0,v_0)$-entry with value $1$ is the only nonzero entry.

    We claim that $\la_1\tpe{F_R, \ell} \ge -\la$ if and only if
    \begin{equation} \label{eqn:matrix-1}
        A_{(F_R,\ell)} + \la I - \left(\frac{2\la-2}{\la^2-1}+\frac{1}{\la}\right)E_{v_0,v_0}
    \end{equation} is positive semidefinite. Indeed, we partition the matrix $A_{\tpesub{F_R,\ell}}+\la I$ into blocks:
    \[\begin{pmatrix}
        A_{(F_R, \ell)} & B \\
        B^\T & C
    \end{pmatrix}\]
    Since $C = A_{\etwographsub} + \la I$ and $\la > 1$, the block $C$ is positive definite. Therefore, the Schur complement $A_{(F_R, \ell)} + \la I - BC^{-1}B^\T$ of $C$ is positive semidefinite. Since the only nonzero row of $B$ is its $v_0$-th row, say $B_{v_0}$, the matrix $BC^{-1}B^\T$ simplifies to $B_{v_0}C^{-1}B_{v_0}^\T E_{v_0,v_0}$. We then compute directly:
    \[
        B_{v_0}C^{-1}B_{v_0}^\T = \begin{pmatrix}
            1 \\ 1 \\ 1
        \end{pmatrix}^\T
        \begin{pmatrix}
            \la & 1 & \\
            1 & \la & \\
            & & \la
        \end{pmatrix}^{-1}
        \begin{pmatrix}
            1 \\ 1 \\ 1
        \end{pmatrix}
        =\frac{2\la-2}{\la^2-1}+\frac{1}{\la}.
    \]
    Similarly, we can prove that $\la_1\ape{F_R,\ell} \ge -\la$ if and only if
    \begin{equation} \label{eqn:matrix-2}
        A_{(F_R,\ell)} + \la I - \left(\frac{\la}{\la^2-1}+\frac{1}{\la}\right) E_{v_0,v_0}
    \end{equation}
    is positive semidefinite. Finally, take $\la = -\la_1\tpe{F_R,\ell} \ge 2$, and observe that the matrix in \eqref{eqn:matrix-2} minus the matrix in \eqref{eqn:matrix-1} is equal to the matrix $(\la-2)/(\la^2-1)E_{v_0,v_0}$, which is positive semidefinite.
\end{proof}

\begin{proof}[Proof of \cref{cor:main}]
    Suppose that $G$ is a connected graph on at least $18$ vertices such that $\la_1(G) \in (-\las, -2)$.
    In view of \cref{thm:main3,thm:main4}, $G$ is either an augmented path extension or a twisted maverick graph. We break the rest of the proof into two cases.
    
    \paragraph{Case 1:} $G$ is an augmented path extension. In view of \cref{thm:reduction,lem:char}, there exists a connected bipartite single-rooted graph $H_r$ and $\ell \in \N$ such that $G$ is $\ape{L(H_r),\ell}$. Let $(u_0, u_1, u_2, u_c)$ be the witness for the augmented path extension. Clearly $u_c$ is a leaf of $G$, and $G - u_c$ is the path extension $(L(H_r), \ell+2)$, which is the line graph of the bipartite graph $(H_r, \ell+3)$. We are left to prove the uniqueness of such a leaf. We claim that there exists a neighbor of $u_0$ in $G - \sset{u_1,u_2,u_c}$ that is not a leaf of $G$. Assume for the sake of contradiction that every neighbor of $u_0$ in $G - \sset{u_1,u_2,u_c}$ is a leaf of $G$. In this case, $\ell = 0$ and $F_R$ is a null graph with more than one vertex, which contradicts the connectedness of $H_r$. Let $u_{-1}$ be such a neighbor of $u_0$. Then $\sset{u_{-1},u_0,u_1,u_c}$ induces a star $S_3$ with $3$ leaves in $G$. Since a line graph cannot contain $S_3$ as a subgraph, and $u_c$ is the only leaf of $G$ in $\sset{u_{-1},u_0,u_1,u_c}$, we have to remove $u_c$ from $G$ to obtain a line graph.

    \paragraph{Case 2:} $G$ is a twisted maverick graph $\tpe{F_R, \ell}$. In view of \cref{lem:tpe-vs-ape}, $\la_1\ape{F_R, \ell} > -\las$. By \cref{thm:reduction,lem:char}, there exists a connected bipartite single-rooted graph $H_r$ such that $F_R$ is $L(H_r)$. Let $(u_0,u_1,u_2,u_c)$ be the witness for the twisted path extension. Clearly $u_c$ is a leaf of $G$, and $G - u_c$ is the graph $(L(H_r),\ell,\ktwo)$, which is the line graph of the bipartite graph $(H_r,\ell+1,\ktwoc)$. The proof for the uniqueness of such a leaf follows exactly that of the previous case.
\end{proof}

\begin{remark}
    The order 18 in \cref{cor:main} is the smallest possible because of the two maverick graphs of order 17 that are not twisted --- each has a unique leaf, and its removal results in a graph that contains the star $S_3$ as a subgraph.
\end{remark}

\section{Beyond the classification theorem of \texorpdfstring{$\G(\las)\setminus \G(2)$}{G(λ*)\G(2)}} \label{sec:beyond}

In this section, we briefly explore graphs whose smallest eigenvalues lie slightly beyond $-\las$. As it turns out, \cref{thm:main1} extends, whereas \cref{thm:main2} does not.

\begin{theorem} \label{thm:beyond-las-yes}
    There exist $\la > \las$ and $N \in \N$ such that for every connected graph $G$ on more than $N$ vertices, if the smallest eigenvalue of $G$ is in $(-\la, -2)$, then $G$ is isomorphic to an augmented path extension of a rooted graph.
\end{theorem}

\begin{theorem} \label{thm:beyond-las-no}
    For every $\la > \las$, every finite family $\F$ of rooted graphs, and every $N \in \N$, there exists a connected graph $G$ on more than $N$ vertices such that the smallest eigenvalue of $G$ is in $(-\la, -\las)$, and $G$ is not an augmented path extension of any rooted graph in $\F$.
\end{theorem}

To understand why \cref{thm:main1} generalizes beyond $-\las$, note that in its derivation in \cref{sec:forb}, the constant $\las$ plays an essential role only in \cref{lem:computation}. In \cite[Appendix~A]{JP25}, Jiang and Polyanskii already noted that $\las \approx 2.01980$ can be replaced with $101/50 = 2.02$ in \cref{lem:computation}. To obtain the best constant, we define the following graphs and compute the limit of their smallest eigenvalues.

\begin{proposition} \label{lem:lambda-prime}
    For every $n \in \N$ with $n \ge 4$, define the graph $E_n'$ as the path extension $(\ftwo, n-4)$ in \cref{fig:ten}. The smallest eigenvalue of $E_n'$ decreases to $-\la'$, where \[\la' := \ga + 1/\ga \approx 2.02124\] and $\ga$ is the unique positive root of $x^4 + x^3 = x^2 + 2$.
\end{proposition}

\begin{figure}
    \centering
    \begin{tikzpicture}[very thick, scale=0.3, baseline=(v.base)]
        \coordinate (v) at (0,0);
        \draw (3,0) -- +(0,2) node[vertex]{};
        \draw (1,2) node[vertex]{} -- (1,0) node[vertex]{} -- (3,0) node[vertex]{} -- cycle;
    \end{tikzpicture}\qquad
    \begin{tikzpicture}[very thick, scale=0.3, baseline=(v.base)]
        \coordinate (v) at (0,0);
        \draw (3,0) -- +(0,2) node[vertex]{};
        \draw (3,0) --  (5,0) node[vertex]{};
        \draw (1,2) node[vertex]{} -- (1,0) node[vertex]{} -- (3,0) node[vertex]{} -- cycle;
    \end{tikzpicture}\qquad
    \begin{tikzpicture}[very thick, scale=0.3, baseline=(v.base)]
        \coordinate (v) at (0,0);
        \draw (3,0) -- +(0,2) node[vertex]{};
        \draw (3,0) --  (5,0) node[vertex]{} -- (7,0) node[vertex]{};
        \draw (1,2) node[vertex]{} -- (1,0) node[vertex]{} -- (3,0) node[vertex]{} -- cycle;
    \end{tikzpicture}\qquad
    \begin{tikzpicture}[very thick, scale=0.3, baseline=(v.base)]
        \coordinate (v) at (0,0);
        \draw (3,0) -- +(0,2) node[vertex]{};
        \draw[dashed] (6.7,0) -- (11,0);
        \draw (3,0) -- (5,0) node[vertex]{} -- (7,0) node[vertex]{};
        \draw (1,2) node[vertex]{} -- (1,0) node[vertex]{} -- (3,0) node[vertex]{} -- cycle;
        \draw (11,0) node[vertex]{} -- (13,0) node[vertex]{} -- (15,0) node[vertex]{};
        \draw [braket] (15,0.1) -- (1,0.1) node [black,midway,yshift=-15pt] {\footnotesize $n-2$};
    \end{tikzpicture}
    \caption{$E_4'$, $E_5'$, $E_6'$ and $E_n'$.} \label{fig:ten}
\end{figure}

\begin{proof}
    Using \cref{lem:lim-smallest-ev}, we know that the matrix
    \[
        \begin{pmatrix}
            r & 1 & 1 & 1 \\
            1 & x & 1 & 0 \\
            1 & 1 & x & 0 \\
            1 & 0 & 0 & x
        \end{pmatrix}
    \]
    is singular when $x = \la'$, where $r = x/2 + \sqrt{x^2/4-1}$. Upon substituting $x = r + 1/r$, the determinant of the above matrix equals $r^4 - r^2 + 2r + 2/r - 2/r^2 - 3$, which factors into $(r^4 + r^3 - r^2 -2)(r^2-r+1)/r^2$. Since $\ga$ is the unique positive root of $r^4 + r^3 - r^2 - 2$, and no root of $r^2 - r + 1$ is real, it must be the case that $\la' = \ga + 1/\ga$.
\end{proof}

Besides $E_6$, the graph $E_6'$ is also one of the 31 minimal forbidden subgraphs for the family $\D_\infty$ of generalized line graphs. It turns out that $E_6'$ serves as the bottleneck case in \cref{lem:computation}.

\begin{proposition} \label{lem:computation-pro}
    For every minimal forbidden subgraph $F$ for the family $\D_\infty$ of generalized line graphs, if $F$ is not isomorphic to $E_6$, then $\lim_{\ell \to \infty}\la_1(F_R, \ell) < -95/47$ for every nonempty vertex subset $R$ of $F$, except when $F_R$ is the rooted graph $E_6'$ in \cref{fig:rooted-e2-prime}.
\end{proposition}

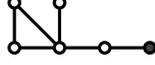
\begin{figure}[t]
    \centering
    \begin{tikzpicture}[very thick, scale=0.3, baseline=(v.base)]
        \coordinate (v) at (0,0);
        \draw (3,0) -- +(0,2) node[vertex]{};
        \draw (3,0) --  (5,0) node[vertex]{} -- (7,0) node[root-vertex]{};
        \draw (1,2) node[vertex]{} -- (1,0) node[vertex]{} -- (3,0) node[vertex]{} -- cycle;
    \end{tikzpicture}
    \caption{The rooted graph $E_6'$.} \label{fig:rooted-e2-prime}
\end{figure}

We postpone the computer-assisted proof of \cref{lem:computation-pro} to \cref{sec:app}. Now we prove \cref{thm:beyond-las-yes} for every $\la \in (\las, \la')$.

\begin{proof}[Proof of \cref{thm:beyond-las-yes}]
    Pick an arbitrary $\la \in (\las, \la')$. Since $\la' \approx 2.02124 < 2.02127 \approx 95/47$, by \cref{lem:lambda-prime,lem:computation-pro}, we may replace the constant $\las$ in \cref{lem:computation} with $\la$. The rest of the proof follows exactly that of \cref{thm:main1} on \cpageref{proof:main1}.
\end{proof}

\begin{remark}
    For $n \ge 10$, the graph $E_n'$ is in $\G(\la') \setminus \G(2)$. Since it is not an augmented path extension, \cref{thm:beyond-las-yes} no longer holds for $\la \ge \la'$.
\end{remark}

To show that it is impossible to generalize \cref{thm:main2} beyond $-\las$, we need the following result on the set of smallest graph eigenvalues.

\begin{theorem}[Theorem~2.19 of Jiang and Polyanskii~\cite{JP25}] \label{thm:dense}
    For every $\la > \las$, there exist graphs $G_1, G_2, \dots$ such that $\lim_{n\to\infty} \la_1(G_n) = -\la$. \qed
\end{theorem}

\begin{proof}[Proof of \cref{thm:beyond-las-no}]
    Fix $\la > \las$. Assume for the sake of contradiction that there exists a finite family $\F$ of rooted graphs and $N \in \N$ such that every graph $G$ on more than $N$ vertices, if $\la_1(G) \in (-\la,-\las)$, then $G$ is isomorphic to an augmented path extension of a rooted graph in $\F$. By \cref{lem:limit}, for each rooted graph $F_R \in \F$, we can set $\la_{F_R} := \lim_{\ell\to\infty}\la_1\ape{F_R,\ell}$. We can then pick an open interval $I \subset (-\la, -\las)$ that avoids the finite set $\dset{\la_{F_R}}{F_R \in \F}$. Note that only finitely many graphs have their smallest eigenvalues in $I$, contradicting \cref{thm:dense}.
\end{proof}

\section{Concluding remarks} \label{sec:remarks}

Since \cref{thm:beyond-las-yes} holds for every $\la \in (\las, \la')$, we ask the following natural question.

\begin{problem}
    Classify all the connected graphs with smallest eigenvalue in $(-\la',-\las)$. In particular, classify such graphs on sufficiently many vertices.
\end{problem}

To conclude the paper, we reiterate a similar problem raised in \cite{JP25} on \emph{signed graphs}, which are graphs whose edges are each labeled by $+$ or $-$. When we talk about eigenvalues of a signed graph $G^\pm$ on $n$ vertices, we refer to its \emph{signed adjacency matrix} --- the $n \times n$ matrix whose $(i,j)$-th entry is $1$ if $ij$ is a positive edge, $-1$ if $ij$ is a negative edge, and $0$ otherwise.

\begin{problem}
    Classify all the connected signed graphs with smallest eigenvalue in $(-\las,-2)$. In particular, classify such signed graphs on sufficiently many vertices.
\end{problem}

Understanding such signed graphs and extending their classification beyond $-\las$ would offer insights into spherical two-distance sets. We refer the reader to \cite[Section~5]{JP25} and \cite{JW25} for the relevant discussion.

\section*{Acknowledgements} We would like to thank the referees for their careful reading of the manuscript and for their constructive comments, which have helped improve the quality of the paper.

\bibliographystyle{plain}
\bibliography{lambda}

\appendix

\section{Computer-assisted proofs} \label{sec:app}

\begin{cproof}[Proof of \cref{lem:seven-eight-vertex,lem:computation-pro}]
    For the exceptional case in \cref{lem:computation-pro} where $F_R$ is the rooted graph in \cref{fig:rooted-e2-prime}, the path extension $(F_R, \ell)$ is just $E_{\ell+7}'$, whose smallest eigenvalue approaches $-\la'$ as $\ell \to \infty$ according to \cref{lem:lambda-prime}. We strengthen the other inequalities in \cref{lem:computation-pro} by replacing $\la'$ with $95/47$.
    
    In view of \cref{lem:lim-smallest-ev}, for each rooted graph $F_R$ considered in \cref{lem:seven-eight-vertex,lem:computation-pro}, to show $\lim_{\ell \to \infty} \la_1(F_R, \ell) < -95/47$, we only need to show that
    \[
        A_{(F_R, 0)} + (95/47)I - (95/94-3\sqrt{21}/94)E_{v_0,v_0}
    \] is not positive semidefinite, where $v_0$ is the vertex in $V(F_R,0)\setminus V(F)$. Since $95/94-3\sqrt{21}/94 > 6/7$, to show that the above matrix is not positive semidefinite, it suffices to show the matrix
    \begin{equation} \label{eqn:rational-matrix}
        A_{(F_R, 0)} + (95/47)I - (6/7)E_{v_0,v_0}
    \end{equation}
    with rational entries is not positive semidefinite.
    
    Our implementation is straightforward. We iterate through the $7$-vertex graphs labeled by \texttt{A1,...,A39} and the $8$-vertex graphs labeled by \texttt{B1,...,B4} in \cref{lem:seven-eight-vertex}, and the minimal forbidden subgraphs, labeled by \texttt{G1,...,G31}, for the family $\D_\infty$ in \cref{lem:computation-pro}. For each graph $F$, we check whether $A_F + (95/47)I$, a principal submatrix of the matrix in \eqref{eqn:rational-matrix}, is positive semidefinite. Since $95/47$ is not an algebraic integer, it cannot be a graph eigenvalue. We instead check whether $A_F + (95/47)I$ is positive definite. If so, we output the nonempty vertex subsets $R$ of $F$, for which the determinant of the matrix in \eqref{eqn:rational-matrix} is nonnegative.
    
    In the output, either $F$ is isomorphic to $E_7$ or $E_6$, or $F_R$ is the exceptional rooted graph in \cref{fig:rooted-e2-prime}. Therefore the output of our program serves as the proof of the strict inequalities in \cref{lem:seven-eight-vertex,lem:computation-pro}.
    
    Our code is available as the ancillary file \texttt{path\_extension.rb} in the arXiv version of this paper. We provide the input as \texttt{path\_extension.txt} for the convenience of anyone who wants to program independently. In the input, each line contains the label of the graph and the string of the form \texttt{u[1]u[2]...u[2e-1]u[2e]}, which lists the edges \texttt{u[1]u[2],...,u[2e-1]u[2e]} of the graph. The first line represents the graph $E_6$, the next 4 lines represent $B_1, \dots, B_4$ in \cref{fig:e-graphs}, and the remaining lines represent the 31 minimal forbidden subgraphs for $\D_\infty$. 
\end{cproof}

\end{document}